\pdfoutput=1

\documentclass[11pt,14paper]{amsart}
\usepackage[margin=17mm]{geometry}
\usepackage[colorlinks]{hyperref}
 
\usepackage{epsfig}

\usepackage{amssymb}
\usepackage{amsmath,amsthm}
\usepackage[all]{xypic}
\usepackage{amsthm}
\newtheorem{theorem}{Theorem}[section]

\newtheorem{corollary}[theorem]{Corollary}

\newtheorem{definition}[theorem]{Definition}
\newtheorem{example}[theorem]{Example}

\newtheorem{lemma}[theorem]{Lemma}

\newtheorem{proposition}[theorem]{Proposition}
\newtheorem{remark}[theorem]{Remark}

\def\polhk#1{\setbox0=\hbox{#1}{\ooalign{\hidewidth
  \lower1.5ex\hbox{`}\hidewidth\crcr\unhbox0}}}

\newcommand{\LL}{\mathrm{LGr}}
\newcommand{\Span}[1]{\mathrm{Span}\left\{#1\right\}}

\def\gh{ \mathfrak h}

\def\gl{ \mathfrak l}

\def\gp{ \mathfrak p}

\def\ker{\mbox{\rm ker \,}}
\def\coker{\mbox{\rm coker \,}}

\newcommand{\vol}{\mathrm{vol}}

\def\R{\mathbb{R}}
\def\p{\mathbb{P}}
\def\O{\mathcal{O}}
\def\E{\mathcal{E}}

\def\LLL{\mathcal{L}}

\def\CC{\mathcal{C}}
            
\def\GL{\mathrm{GL}}    \def\gl{\mathfrak{gl}}      
\def\gl{\mathfrak{gl}}
\def\sp{\mathfrak{sp}}
\def\Sp{\mathrm{Sp}}

\def\Hom{\mathrm{Hom}\,}

\def\I{\mathrm{I}}
\def\II{\mathrm{II}}
\def\Ker{\mathrm{Ker}}

\newcommand{\Smbl}{\mathrm{Smbl}}
\newcommand{\virg}[1]{``#1''}
\let\epsilon\varepsilon

\let\L\LL

\title{Completely exceptional $2^\textrm{nd}$ order PDEs via  conformal   geometry and BGG resolution}

   \author{Jan Gutt}

  \address{Center for Theoretical Physics, Polish Academy of Sciences, Al. Lotnikow 32/46, 02--668 Warsaw, POLAND} \email{jgutt@cft.edu.pl}
 \author{Gianni Manno} 

   \address{Dipartimento di Matematica ``G. Lagrange'', Politecnico di Torino, Corso Duca degli Abruzzi, 24, 10129 Torino, ITALY.} \email{d027258@polito.it}
 \author{Giovanni Moreno} 

   \address{IMPAN, ul. Sniadeckich 8, 00--656 Warsaw, POLAND.} \email{gmoreno@impan.pl}
\date{\today}
 
\begin{document}

\begin{abstract}
By studying the development of shock waves out of  discontinuity waves, in 1954 P. Lax discovered a class of PDEs, which he called ``completely exceptional'', where such a transition does not occur after a finite time. A   straightforward integration of the completely exceptionality conditions allowed Boillat  to show  that such PDEs are actually of Monge--Amp\`ere type. In this paper,  we first recast these conditions  in terms of characteristics,  and then we show that the completely exceptional PDEs, with 2 or 3 independent variables, can be  described in terms of the conformal geometry of the Lagrangian Grassmannian, where they are naturally embedded. Moreover, for an arbitrary number of independent variables, we show that the space of $r^\textrm{th}$  degree sections of the Lagrangian Grassmannian can be resolved via a BGG operator. In the particular case of $1^\textrm{st}$ degree sections, i.e.,  hyperplane sections or, equivalently,  Monge--Amp\`ere equations, such operator is a close analog  of the trace--free second fundamental form.
\end{abstract}

%
%
%


 \maketitle

\setcounter{tocdepth}{2}
 \tableofcontents

\section{Introduction}\label{subsecSupersonica}

According to P. Lax (see \cite{MR0481580,Boillat_et_al,MR0068093,RuggeriWave}),   completely exceptionality
is what marks the line between \virg{genuinely nonlinear} systems  and nonlinear systems \virg{displaying a linear behavior}.\par
The essence of this  claim can be   grasped by   thinking about a   well--known physical phenomenon:  the \virg{sonic boom}. An aircraft produces a disturbance in the otherwise steady surrounding medium---the air. In the terminology of \cite{RuggeriWave}, the solution representing the \virg{still air} is referred to as the \emph{unperturbed state} $u_0$, whereas the \emph{perturbed state} $u_1$ describes the fluctuations in the air pressure introduced by the flying object. Since the disturbance travels at finite speed, there is always a portion of the medium which does not feel the fluctuations: this implies the existence of the so--called \emph{wave front}, i.e., the locus where the solution $u_0$ \emph{bifurcates} into the solution $u_1$. In other word, the \emph{global solution}, which describes the status of the medium  traversed by a sub--sonic vehicle, is the \emph{discontinuity wave} obtained by gluing $u_0$ from the left of the wave front, with $u_1$ on its right (see figure below). By construction, a discontinuity wave is a $C^1$ piecewise solution, even if both $u_0$ and $u_1$ are regular (i.e., $C^\infty$).\par
\noindent\epsfig{file=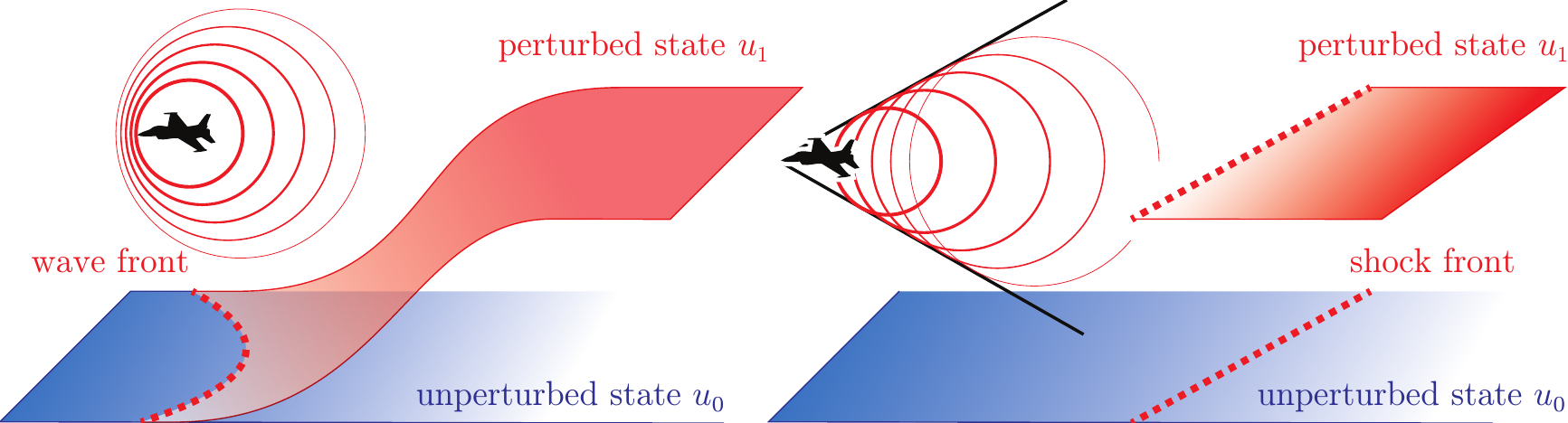,width=\textwidth}\par
As the plane goes super--sonic, a geometric locus appears, where the pressure peaks of the disturbances introduced by the moving body sum up to an explosive value. Immediately beyond this locus, the pressure drops back to its steady--state value. In other words, the states $u_0$ and $u_1$ are now separated, i.e., the global solution obtained, as before, by gluing\footnote{It should be stressed, that not all pairs of regular solutions can be glued to a \emph{weak solution} (see  \cite{RuggeriWave},  Def. 3.1), but some compatibility conditions, known as Rankine--Hugoniot equations, must be checked beforehand (see  \cite{RuggeriWave},  Sec. 3.2.2).} $u_0$ and $u_1$, is not even continuous, as its value suffers a jump along what is now called a \emph{shock front}. Such a solution is known as a \emph{shock wave}, the name \virg{shock} being evocative of the fact that, when the geometric locus intersects the ground, people can actually feel the pressure discontinuity as a loud, sometimes painful, bang.\footnote{Needless to say, a non--continuous solution is just an approximation, since a non--zero jump in the field value can only be caused  by an  \virg{infinite force}, whose existence is usually forbidden by physical laws. In practice, the jump discontinuity in a sonic boom takes place in a very thin (though three--dimensional) domain, whose thickness has been experimentally detected in
$\sim$ 200 nm \cite{9780471742999}.}\par
A key feature of both discontinuity and shock waves is that the wavefront, in the former, as well as the shock front, in the latter, are \emph{characteristic hypersurfaces} (or \emph{lines}, when the number of independent variables is equal to $2$). The speed at which the characteristic lines travel in time is known as the \emph{characteristic speed}---the speed of sound, in our case.

Essentially, a system is called \emph{completely exceptional} if the  phenomenon just described  never occurs, i.e., when
discontinuity waves never turn into shock waves \cite{Ruggeri01061978}. Since the development of shock waves out of discontinuity waves is a typical feature of nonlinear systems of PDEs, it is easy to understand P. Lax's claim that   non--completely exceptional systems are, in a sense, \virg{genuinely nonlinear}.
The notion of complete exceptionality was extensively studied in the case of quasi--linear hyperbolic systems of first order PDEs (see for instance \cite{MR0481580,Boillat_et_al,Ruggeri01061978}). The eigenvalues associated to a given system of this type (that can be interpreted as characteristic velocities) can have jump discontinuities along a characteristic hypersurface, as higher (in this case, first) order derivatives of solutions can suffer a jump through it \cite{Boillat_et_al}: it can be proved that this phenomenon does not occur if and only if the system is completely exceptional \cite{Ruggeri01061978}. These reasonings naturally extend to $2^\textrm{nd}$ order (scalar) PDEs, so that we can give the following definition.
\begin{definition} \label{defCompExcept}
A quasi--linear hyperbolic system of first order PDEs (resp. hyperbolic $2^\textrm{nd}$ order (scalar) PDE) is \emph{completely exceptional} if its characteristic velocities have no jump discontinuities.
\end{definition}
We underline that for characteristic velocities of a $2^\textrm{nd}$ order scalar PDE we mean the solutions of its characteristic polynomial.
On the top of that (as pointed out by G. Boillat, A. Donato and others),
any completely exceptional  $2^\textrm{nd}$ order scalar PDE must be of Monge--Amp\`ere type.
\begin{proposition}[\cite{MR1139843,DonatoRamgulamRogers}]\label{eq.Boillat.MAE}
A hyperbolic scalar $2^{\textrm{nd}}$ order PDE is completely exceptional iff it is a (hyperbolic) Monge--Amp\`ere equation.
\end{proposition}
The first aim of this paper is to provide a geometric   notion   of complete exceptionality, given  in terms of   characteristics, rank--one lines, and symbols.
More precisely,    the characterisation  of completely exceptional PDEs can be reduced to the characterisation of hyperplane sections in the Lagrangian Grassmannian  $\LL(n,2n)$: for $n=2$ we show that the problem reduces to the characterisation of flat surfaces in a three--dimensional conformal space known as the Lie quadric, whereas for $n=3$ one needs to study the analogues of flat five--folds in a six--dimensional space equipped with a trivalent conformal structure.\par

For an arbitrary number $n$ of independent variables, we focus on the function $F$ instead of the equation $\E=\{F=0\}$. More precisely,  we provide a general method to detect whether   a function $F$,  to which we associate the  PDE $\E:=\{F=0\}$,   gives rise to a  completely exceptional equation. Such a method is purely representation--theoretic, being based on the homogeneous structure of $\LL(n,2n)$ and a corresponding BGG sequence.\par Indeed, as a compact homogeneous space  of the Lie group $\Sp(n,\R)$, the manifold $\LL(n,2n)$ is canonically embedded into a suitable projective space $\p^M$ (see Remark \ref{remMinimality} later on).  The reason why one must bring in the BGG operator, is that is represents the most natural way  to resolve the space of functions on $\LL(n,2n)$ which are the restrictions of the $r^\textrm{th}$ degree  homogeneous polynomials on $\p^M$. For $r=1$ one has hyperplane sections, i.e., Monge--Amp\`ere equations. For $r>1$ one encounters new interesting classes of nonlinear PDEs, like, e.g., the ``quadratic   Monge--Amp\`ere equations'', currently studied in  \cite{ArticoloConPawelEKatia}. According to Proposition \ref{eq.Boillat.MAE}, the  complete exceptionality  must be regarded as   a form of ``linearity'', in the sense that such a linearity reveals itself only through the  embedding of the Lagrangian Grassmannian  into $\p^M$.  The last section of this paper points towards the existence of ``higher--order'' analogues of complete exceptionality, i.e., a property characterising    quadratic, cubic, quartic, etc., Monge--Amp\`ere equations.\par

The notions and the results herewith obtained can be generalised to the higher--order case, in the spirit of   \cite{MR1194520,ArticoloAntipatetico}, but this  will be the subject of a future work.

\subsection*{Structure of the paper and main results}
Section \ref{secBackground} merely puts together all the necessary geometric gadgets, and can be safely skipped by anyone feeling confident enough.
In Section \ref{secCompExc} we explain the property of complete exceptionality of a $2^\textrm{nd}$ order PDE: first, we go through  the classical derivation,  and then we recast it in a more intrinsic geometric way.
In this new perspective, Theorem \ref{th.main.1} of Section \ref{secCentrale} proposes a test to check whether a given PDE is completely exceptional (i.e., in view of Proposition \ref{eq.Boillat.MAE}, a Monge--Amp\`ere equation), based on standard considerations of conformal geometry (Proposition \ref{prop.hyperplane}); this test is then generalised to three--valued conformal structures (Section \ref{secL36}). Finally, in Section \ref{sec.Jan}, we show that the functions whose zero loci correspond to a $r^\textrm{th}$ degree  hypersurface of $\L(n,2n)$ constitute    the kernel of suitable natural differential operator arising in a purely representation--theoretic approach to the problem (Theorem \ref{thm:final-bgg}), obtaining the case of Monge--Amp\`ere equations, i.e. of hyperplane sections, in the particular case $r=1$.

\subsection*{Notations and conventions}
For simplicity, when $X$ is a vector field and $\mathcal{P}$ is a distribution on the same manifold, we write ``$X\in\mathcal{P}$'' to mean that $X$ is a smooth (local) section of tangent subbundle $\mathcal{P}$.
We stress that, when dealing with the classical notion of complete exceptionality, we confine ourselves to the two--dimensional context, since the multidimensional case (see \cite{MR571041}) poses no extra conceptual difficulties with respect to the two--dimensional one, the only differences being formal.

\section{Geometric background}\label{secBackground}

\subsection{Prolongation of contact manifolds, Lagrangian Grassmannians and 2--nd order nonlinear PDEs}\label{subLagEq}
The framework for PDEs based on the prolongations of contact manifolds was implicitly made use of in the original  Cartan's works.  One of the earliest modern treatment of the subject was given  by Yamaguchi in 1982 \cite{MR722524}. For the present purposes, we stick to the simplified setting and on the notation used  in \cite{MR2985508}. For   more details on the topics, and extensive references, the reader should consult the classical book \cite{MR2352610}.\par  For the sake of self--consistency, the main ideas ad results   are briefly recalled below.\par

Let $(M,\mathcal{C})$ be a \emph{$(2n+1)$--dimensional contact
manifold}, i.e., a $(2n+1)$-dimensional manifold where $\mathcal{C}$ is a completely non--integrable distribution of codimension $1$. Locally, $\mathcal{C}$ is the kernel of (a contact) $1$--form $\theta$ (defined up to a conformal factor). By Darboux Theorem, there exists a system of coordinates  $(x^i,u,p_i)$, $i= 1,\dots, n$ (which we call \emph{contact coordinates}) such that $\theta = du - p_i dx^i$.
The restriction $\omega := d \theta|_{\mathcal{C}}$
defines on each hyperplane $\mathcal{C}_m$ a conformal symplectic structure: Lagrangian (i.e., maximally $\omega$--isotropic) $n$--dimensional planes of $\mathcal{C}_m$ are tangent to maximal integral submanifolds of $\mathcal{C}$. We denote by $\L(\mathcal{C}_m)$ the \emph{Grassmannian of Lagrangian planes} of $\mathcal{C}_m$  and by
\begin{equation}\label{eq.pi}
\pi: M^{(1)} = \bigcup_{m \in M}\L(\mathcal{C}_m) \to M
\end{equation}
the  bundle of Lagrangian planes.
Previous constructions lead naturally to consider a geometric object which will play an important role in our analysis, namely  the so--called   \emph{tautological bundle} over $M^{(1)}$
\begin{equation}\label{eq.taut}
L\rightarrow  M^{(1)}\, ,
\end{equation}
where   the fibre $L_{m^1}$ is $m^1\in M^{(1)}$ itself, understood as an $n$-dimensional subspace of $\mathcal{C}_{\pi(m^1)}$.
Contact coordinates   on $M$ induce coordinates on $M^{(1)}$: a point $m^1 \equiv L_{m^1} \in M^{(1)}$
has coordinates $(x^i, u, p_i, p_{ij})$, $1 \leq i \leq j \leq n$ iff the
corresponding Lagrangian plane $L_{m^1}$ is given by $m^1\equiv L_{m^1} = \langle D_{x^i} \rangle_{i=1\dots n}$,
where
\begin{equation}\label{eq.tot.der}
D_{x^i}\overset{\textrm{def}}{=} \partial_{x^i} + p_i \partial_u + p_{ij}\partial_{p_j}\, ,
\end{equation}
and  $\|p_{ij}\|$ is a symmetric matrix.

\smallskip\noindent
A \emph{scalar $2^{nd}$ order PDE with $n$ independent variables and one unknown function} is defined as a
hypersurface $\mathcal{E}$ of $M^{(1)}$ and its \emph{solutions} are Lagrangian submanifolds $\Sigma \subset M$ such that $T\Sigma \subset \mathcal{E}$. We assume that $\E$ projects onto $M$: this means that  the restriction   $\pi|_\mathcal{E}$   is a bundle over $M$ whose fibre at $m\in M$ is
\begin{equation}\label{eq.principal}
\mathcal{E}_m:=\mathcal{E}\cap \L(\mathcal{C}_m)\, .
\end{equation}
This eventually explains our perspective:  $\mathcal{E}_m$ is a hypersurface of the Grassmannian  $\L(\mathcal{C}_m)$ of Lagrangian planes of $\mathcal{C}$. In what follows,   different---yet equivalent---symbols will be used for  $\L(\mathcal{C}_m)$: either $\L(n,2n)$, if we wish to detach the notion from the particular symplectic space $\CC_m$, or $X_n$, when we focus in the projective--algebraic aspects of the space, as in the last Section \ref{sec.Jan}.
We conclude by recalling that the  conformal symplectic group $\mathrm{CSp}(\CC_m)$ acts naturally  and transitively on $\L(\mathcal{C}_m)$, which turns out to be a homogeneous manifold. In particular,  the abstract fibre of $M^{(1)}$ can be described as the homogeneous space   $\L(n,2n)=\Sp(2n,\R)/P$, where $P$ is the parabolic subgroup $  \GL_n \ltimes S^2\R^{n*}$.

\subsection{Natural conformal structures on Lagrangian Grassmannians}
It is well known (keeping in mind the definition of tautological bundle \eqref{eq.taut}) that
\begin{equation}\label{eq.iso.solito}
T\L(\CC_m)\simeq S^2L^*\, ,
\end{equation}
so that   the \emph{rank} of a vector in $T\L(\CC_m)$ is defined as the rank of its correspondent symmetric bilinear form through the isomorphism \eqref{eq.iso.solito}. Of course, this definition is invariant under a conformal change of the symplectic form, and proportional tangent vectors have the same rank, so that this notion applies to directions in $T\L(\CC_m)$ as well. Let
$
T^k\L(n,2n):=\{v\in T\L(n,2n) \,\,|\,\, \text{rank}(v)=k\}
$
be the set of vectors of rank $k$.\par
Thus, we have a canonical distribution of cones on $\L(n,2n)$, made from lines of non--maximal rank: there exists (up to a conformal factor) a unique symmetric $n$-form $T_n$ on $\L(n,2n)$ such that the aforementioned cones are its isotropic varieties. The conformal class of $T_n$ can be represented by the total polarisation of the determinant
$
\det\in S^n(S^2L)\, ,
$
which in turn spans the unique one--dimensional irreducible $\GL(n)$--submodule of the   kernel of the ``total symmetrization'' map:
\begin{equation}\label{eq.symmetrization}
\mathcal{S}:S^n(S^2L)\to S^{2n}L\, .
\end{equation}
Having defined a canonical (conformal) $n$--tensor $T_n$ on $\L(n,2n)$,  we can define the following set of symmetric $2$--tensors:
\begin{equation*}
S^{2}_{T_n}:=\left\{X_{n-2}\,\lrcorner\, T_n\,,\quad X_{n-2}\in \bigotimes^{n-2}T\L(n,2n)\right\}\, .
\end{equation*}
Note that, in the case $n=2$, $X_{0}\lrcorner T_2=X_0\cdot T_2$, with $X_0\in C^\infty(\L(n,2n))$, so that $S^{2}_{T_2}$ is the set of bilinear forms on $\L(2,4)$ that are conformal to $T_2$.\par
The image of the contraction map
$S^{n-2} S^2L^\ast \stackrel{T_n}{\longrightarrow} S^2S^2L$, determined by $T_n$,   is precisely $S^{2}_{T_n}$. On the top of that, the total symmetrization gives rise to the short exact sequence
\begin{equation*}
\ldots\longrightarrow S^{n-2} S^2L^\ast \stackrel{T_n}{\longrightarrow} S^2S^2L\stackrel{\mathcal{S}}{\longrightarrow} S^4L \longrightarrow 0\, ,
\end{equation*}
so that
$S^{2}_{T_n}=\ker\mathcal{S}$. Easy computations shows that
$
\dim S^{2}_{T_n} =\frac{n^2(n^2-1)}{12}=1,6,20, 50, \ldots
$,
for $n=2,3,4,5,\dots$. In other words, $(\L(n,2n),T_n)$ is a genuine conformal manifold only for $n=2$.

\subsection{Prolongation of subspaces and characteristics}\label{secProlong}
The   \emph{prolongation}
$U^{(1)}\subset\L(\CC_m)$ of a subspace $U \subset \CC_m$ is defined as follows
\begin{equation*}
U^{(1)}:=\left\{
\begin{array}{c}
L_{m^1}\in \L(\CC_m)\,\,|\,\, L_{m^1}\supseteq U,\,\,\text{if}\,\,\dim(U)\leq n\, ,\\
\\
L_{m^1}\in \L(\CC_m)\,\,|\,\, L_{m^1}\subseteq U,\,\,\text{if}\,\,\dim(U)\geq n\, .
\end{array}
\right.
\end{equation*}
Since $L=L^\perp$, one  can easily check that $U\subset W \Longrightarrow U^{(1)}\supset W^{(1)}$ and that $U^{(1)}=\left({U^\perp}\right)^{(1)}$.

\begin{definition}\label{def.cono.car.point}
The set
$
T^1_{m^1}\E:=
T_{m^1}\E\cap T^1_{m^1}\L(\CC_m)
$ 
of rank--one tangent vectors to $\E$
is called the \emph{rank--one cone} at $m^1$ of the hypersurface $\E$.
\end{definition}

\begin{proposition}\label{prop.char.isotropic}
Vectors of $T^1_{m^1}\E$ are, up to sign, the tensor squares $\eta \otimes \eta$ of a covectors $\eta \in  L_{m^1}^*$ such that $\eta \otimes \eta$, understood as a vertical vector to $M^{(1)}$, is tangent to $\E$.
\end{proposition}

\begin{definition}\label{def.char.point}
An isotropic subspace $U$ is called \emph{characteristic} for a covector $\rho\in T^*_{m^1}\L(n,2n)$ if $U\subset L_{m^1}$ and
$\rho|_{T_{m^1}U^{(1)}}=0$. It is called characteristic for a
hypersurface $\{F=0\}$ of $\L(n,2n)$ at a point
$m^1$ of the hypersurface if it is characteristic for $(dF)_{m^1}$. A covector $\eta\in L_{m^1}^*$ is called characteristic for $\rho$ if $\Ker(\eta)$ is characteristic for $\rho$.
\end{definition}

Characteristic directions and characteristic subspaces turn out to be in a tight relationship.
\begin{remark}\label{rem.rank1.and.hyperplane}
Any rank--one vector $v=\pm \eta \otimes \eta
\in T^1_{m^1}\L(n,2n)$ defines the hyperplane $H=\Ker(\eta)$ of $L_{m^1}$ which has the property that $T_{m^1}H^{(1)}=\langle v \rangle$, and
viceversa (note that $H^{(1)}$ is $1$--dimensional). Thus we have the key correspondence between hyperplanes of $L_{m^1}$ (which correspond  to  elements of
$\mathbb{P}L_{m^1}^*$) and  rank--one directions of
$T_{m^1}\L(n,2n)$.
Accordingly,   $(n-1)$--dimensional characteristic subspaces for a hypersurface $\E$  of $\L(n,2n)$ at $m^1$ are in one--to--one correspondence with the characteristic directions for $\E$ at $m^1$.
\end{remark}

We focus now on the   case $n=2$, which will be of key importance in the sequel. Fix a point $m^1\in M^{(1)}$ and let $\ell_m\subset L_{m^1}$ be a direction in $L_{m^1}$, i.e., an element of $\p L_{m^1}$, and
\begin{equation}\label{eqDefEllEmme}
\ell_m= [\lambda:-1]=\Span{ \lambda D_{x^1}- D_{x^2}}=\ker(dx^1+\lambda dx^2)\, ,
\end{equation}
where   $D_{x^i}$ are defined by \eqref{eq.tot.der}.
\begin{proposition}\label{propComeEFattoIlRaggio}
The tangent line to $\ell_m^{(1)}$, where $\ell_m$ is defined by \eqref{eqDefEllEmme}, is given by
$
T_{m^1}\ell_m^{(1)}=$\linebreak $\Span{\partial_{p_{11}}+\lambda \partial_{p_{12}}+\lambda^2 \partial_{p_{22}}}
$. 
\end{proposition}
\begin{proof}
 First describe $\ell^{(1)}$ as the set of $p_{ij}$ such that
 \begin{equation*}
\left(\begin{array}{cccc}\lambda&-1 & 0 & 0 \\1 & 0 & p_{11} & p_{12} \\0 & 1 & p_{12} & p_{22}\end{array}\right)
\end{equation*}
is of rank $\leq 2$. Easy computations shows that this is the case if and only if
\begin{equation}\label{eqParamCurvVert}
(p_{11},p_{12},p_{22})=(t,\lambda t,\lambda^2 t)\, ,\quad t\in\R\, .
\end{equation}
\end{proof}

\subsection{Pl\"ucker embedding}\label{sec.Plucker}
An important feature of the Lagrangian Grassmannian $\LL(n,2n)$ is that it can be naturally embedded into a projective space, by taking the ``volume'' of its elements, viz.
\begin{eqnarray}\label{eq.Plucker}
\LL(n,2n)  &\longrightarrow & \p(\Lambda^n\R^{2n})\, ,\\
L=\Span{l_1,\ldots,l_n} &\longmapsto & \vol (L):=[l_1\wedge\cdots\wedge l_n]\, .\nonumber
\end{eqnarray}
By \emph{hyperplane sections} of $\LL(n,2n)$ we mean the outcomes of the intersection of $\LL(n,2n)$ with a hyperplane of $\p(\Lambda^n\R^{2n})$ via the above embedding, that is called \emph{Pl\"ucker embedding}. It has the remarkable property of converting the \emph{curves} $U^{(1)}$, where $U\subset L\in \LL(n,2n) $ is a hyperplane, into \emph{lines} into its image.
\begin{example}\label{esFinale}
 Let $n=2$, and consider the curve $\ell^{(1)}\subset\LL(2,4)$ given by the Lagrangian planes (see \eqref{eqParamCurvVert})
 \begin{equation*}
L(t)= \left(\begin{array}{cc}t & t\lambda \\t\lambda & t\lambda^2\end{array}\right)\, .
\end{equation*}
 Then $\vol (L(t))=[1:t:t\lambda:t\lambda^2:0]$ is precisely the line joining the points $[1:0:0:0:0]$ and $[0:1:\lambda:\lambda^2:0]$.
\end{example}
\begin{remark}\label{remMinimality}
 Actually, the target space $ \p(\Lambda^n\R^{2n})$ is ``oversized'', for there exists a proper subspace $\p^M\subset \p(\Lambda^n\R^{2n})$, characterised by the property of being the smallest one containing the image of the map \eqref{eq.Plucker}. The existence of such a  $\p^M$  is a basic representation--theoretic fact (see Section \ref{sec.Jan} later on). It is worth observing that the notion of a hyperplane section stays unaltered.
\end{remark}

\subsection{Symbol of a function and its iterations}

Let $F\in C^\infty(M^{(1)})$. Define the \emph{symbol} $\Smbl(F)$ of $F$ as the map
$
m^1\in M^{(1)}\to \Smbl(F)_{m^1}:=\big(d(F|_{\L(\CC_m)})\big)_{m^1} \in  T^*_{m^1}\L(\mathcal{C}_m)= T^*_{m^1}\L(n,2n) \simeq S^2 L_{m^1}
$.
By recalling the definition of tautological bundle \eqref{eq.taut}, in short we can write
$
\Smbl(F)\in S^2L
$.
Essentially, the symbol of a function $F\in C^\infty(M^{(1)})$ is its vertical differential with respect to the projection \eqref{eq.pi}. We call the \emph{characteristic equation} the quadratic polynomial associated to $\Smbl(F)$ equated to zero, i.e.
\begin{equation}\label{eq.characteristic}
\sum_{i\leq j}F_{u_{ij}}\xi_i\xi_j=0\, .
\end{equation}
Notice that
$
(d\,\Smbl(F))_{{m^1}}\,:\,\,T_{{m^1}}\L(n,2n) \simeq S^2L^*_{m^1}   \longrightarrow  T_{\Smbl(F)_{m^1}}   S^2 L_{m^1} \simeq S^2 L_{m^1}
$, 
 i.e., the differential of $\Smbl(F)$ takes values in $S^2 L$,
so that $d\,\Smbl(F) \in S^2 L\otimes S^2 L$. Actually a direct computation shows that
$
d\,\Smbl(F) \in S^2 L\odot S^2L
$. 
In contact coordinates,
\begin{equation*}
d\,\Smbl(F) = \sum_{i\leq j,\, h\leq k}F_{u_{ij}u_{hk}}(D_{x^i}\odot D_{x^j})\otimes (D_{x^h}\odot D_{x^k}) = \sum_{i\leq j,\, h\leq k}\left(2-\delta^{(i,j)}_{(h,k)}\right)F_{u_{ij}u_{hk}}(D_{x^i}\odot D_{x^j})\odot (D_{x^h}\odot D_{x^k})\, ,
\end{equation*}
where $D_{x^i}$ are defined by \eqref{eq.tot.der} and $\delta^{(i,j)}_{(h,k)}$ is equal to $1$ if $(i,j)=(h,k)$ and $0$ otherwise.\par
Now we can  define $$\Smbl^2(F):=\mathcal{S}(d\,\Smbl(F))\in S^4L$$ as the projection of $d\,\Smbl(F)$ on $S^4L$ via 
\eqref{eq.symmetrization}.

\begin{remark}\label{remSimbCoord}
According to the general definition, one can obtain a coordinate description of  $\Smbl^2(F)$, for $F\in C^\infty(M^{(1)})$. First, if
\begin{equation}\label{eqSmbl}
 \Smbl (F)=F_{p_{11}} \xi^2+F_{p_{12}} \xi\eta+F_{p_{22}} \eta^2\, ,
\end{equation}
then
$
 \Smbl^2 (F)=\Smbl(F_{p_{11}}) \xi^2+\Smbl(F_{p_{12}}) \xi\eta+\Smbl(F_{p_{22}}) \eta^2
$ 
 turns out to be a degree--four homogeneous polynomial on $L$.
\end{remark}
Iteratively, we can define the $k^\textrm{th}$ order differential $d^k\,\Smbl(F)\in S^{k+1}(S^2L)$ of the symbol $\Smbl(F)$ of $F$ and obtain the map $\Smbl^k(F)\in S^{2k+2}L$.

\section{Complete exceptionality}\label{secCompExc}

\subsection{The equation of completely exceptionality: classical derivation}

Concerning scalar PDEs with $2$ independent variables, the roles of characteristic velocities are played by the roots $\lambda= \xi_2/\xi_1$ of the characteristic equation \eqref{eq.characteristic}:
\begin{equation}\label{eq.symb.lambda}
F_{p_{11}}+F_{p_{12}}\lambda+F_{p_{22}}\lambda^2=0\, .
\end{equation}
The denomination \virg{characteristic velocity} is somehow misleading, since it is \emph{not} the speed at which the wave front travels through the space of independent variables (which would require an additional metric structure), but rather the rate of change of the \emph{tangent direction} to the characteristic lines within the tangent spaces. Even if this correct interpretation may seems dependent on the choice of a local coordinate system,  we shall prove later that it is in fact a truly intrinsic notion. For the time being, we stick to the classical understanding of the characteristic velocity as the \virg{slope} of the characteristic line, at each point of a given solution (see Fig. below).\par
\noindent\epsfig{file=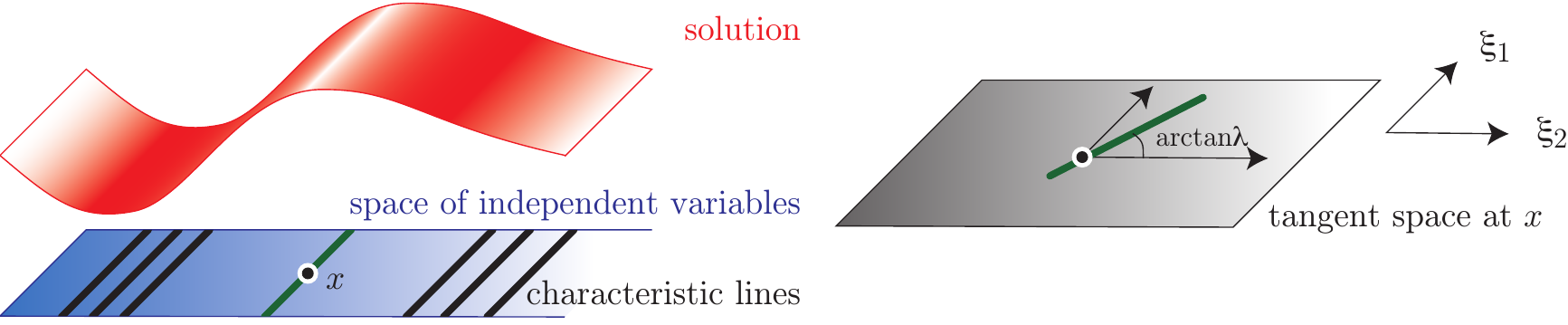,width=\textwidth}
So, we can say that the function $F$ determines the completely exceptional  PDE $\E:=\{F=0\}$    if and only if $\delta(\lambda)=0$ for any solution of the equation \eqref{eq.symb.lambda} restricted to $\E$, where, following \cite{MR1292999,MR1606791,Ruggeri01061978},
\begin{equation*}
\delta:=\left(\frac{\partial}{\partial\phi}\right)_{\phi=0^+} - \left(\frac{\partial}{\partial\phi}\right)_{\phi=0^-}
\end{equation*}
is the operator   measuring the jump of the value of a function $f=f(x^1,x^2)$ through the curve $\{\phi(x^1,x^2)=0\}$.\par
Here some considerations are in order. If we want to compute the jump of a function $f$ through the curve $\{\phi(x^1,x^2)=0\}$ near a point $p$ of such curve, we can compute $\lim_{t\to 0^+}f(\gamma(t)) - \lim_{t\to 0^-}f(\gamma(t))$, with $\gamma(t)$ a curve in the $(x^1,x^2)$--space such that $\gamma(0)=p$ and $\dot{\gamma}(0)$ is transverse to $\{\phi(x^1,x^2)=0\}$.
As recalled in the Introduction,  complete exceptionality is the property which prevents the discontinuity waves from evolving into shock waves. Mathematically, this means that the worst kind of discontinuity which can take place is a  jump one. In turn, this implies   that   the previous value is independent of the chosen curve and, in particular, we can consider an orthogonal curve at $p$.

In view of the above reasoning, to compute $\delta(p_{ij})$ it is enough to take into account the components of $p_{ij}$ along a normal vector of $\phi(x^1,x^2)=0$. For our convenience, we put $p_1=u_{x^1}$, $p_2=u_{x^2}$, $p_{11}=u_{x^1x^1}$, and so on. We need a change of coordinates
\begin{equation}\label{eq.changing}
x^1=x^1(\tau,\mu)\,, \quad x^2=x^2(\tau,\mu)\, ,
\end{equation}
where $\big(x^1(\tau,0),x^2(\tau,0)\big)$ describes the curve $\phi(x^1,x^2)=0$ and $\big(x^1(0,\mu),x^2(0,\mu)\big)$ is a curve orthogonal to $\phi(x^1,x^2)=0$ at $(0,0)$. We can always choose   \eqref{eq.changing} in such a way that
$
(x^1_\tau,x^2_\tau)=(-\tilde{\phi}_{x^2},\tilde{\phi}_{x^1})$, $
(x^1_\mu,x^2_\mu)=(\tilde{\phi}_{x^1},\tilde{\phi}_{x^2})
$,
where
$
(\tilde{\phi}_{x^1},\tilde{\phi}_{x^2})=\frac{\nabla(\phi)}{|\nabla(\phi)|}
$.
A straightforward computation gives
$
\delta(u_{x^ix^j})= \tilde{\phi}_{x^i}\tilde{\phi}_{x^j}\delta(u_{\mu\mu})
$, 
thus getting the Hadamard's relation (see \cite{MR1194520})
$
\delta(u_{x^1x^2})=\left(\frac{\phi_2}{\phi_1}\right)\delta(u_{x^1x^1})$, $ \delta(u_{x^2x^2})=\left(\frac{\phi_2}{\phi_1}\right)^2\delta(u_{x^1x^1})
$.
So, the complete exceptional condition $\delta(\lambda)=0$, for $\lambda$ satisfying \eqref{eq.symb.lambda}, translates as follows:
\begin{equation}\label{eqComPExpectConSegnoSbagliato}
\lambda_{p_{11}}+\lambda_{p_{12}}\lambda+\lambda_{p_{22}}\lambda^2=0\, .
\end{equation}
A direct computation shows that the   condition \eqref{eqComPExpectConSegnoSbagliato}, for a hyperbolic PDE $\{F=p_{22}-h=0\}$, where $h=$\linebreak $h(x^1,x^2,u,p_1,p_2,p_{11},p_{12})$, is described by the following system of PDEs:
\begin{equation}\label{eq.system.Pi0S.2}
\left\{
\begin{array}{l}
h_{p_{11}p_{11}}+h_{p_{11}}h_{p_{12}p_{12}}=0\, ,
\\
\\
2h_{p_{11}p_{12}} + h_{p_{12}}h_{p_{12}p_{12}}=0\, .
\end{array}
\right.
\end{equation}

\begin{proposition}\label{propPrimaEquivalenzaCompEcc}
A hyperbolic PDE $\E=\{F=0\}$ is completely exceptional (Definition \ref{defCompExcept}) if and only if \eqref{eqComPExpectConSegnoSbagliato} is satisfied for all roots $\lambda$ of \eqref{eq.symb.lambda}.
\end{proposition}

\begin{remark}\label{rem.par.equ.sempre.ecc}
For parabolic equations (i.e., $2^\mathrm{nd}$ order PDE whose symbol is a perfect square) condition \eqref{eqComPExpectConSegnoSbagliato} is always fulfilled. In fact, the discriminant of the equation \eqref{eq.symb.lambda} is
$ 
\Delta=F_{p_{12}}^2-4F_{p_{11}}F_{p_{22}}
$, 
and when it is zero (in a neighborhood) the only solution to \eqref{eq.symb.lambda} is $\lambda=-\frac{1}{2}\frac{F_{p_{12}}}{F_{p_{22}}}$. If we substitute this $\lambda$ into \eqref{eqComPExpectConSegnoSbagliato} we obtain the equation
$
\Delta_{p_{22}}F_{p_{12}}F_{p_{22}}-\Delta_{p_{12}}F_{p_{22}}^2-\Delta(F_{p_{22}p_{22}}F_{p_{12}}-F_{p_{22}}F_{p_{12}p_{22}})=0
$
that, in this case, is always satisfied as $\Delta$ is zero.
\end{remark}

\begin{remark}
For the sake of completeness, we  also recall   that the completely exceptional PDEs form a sub--class of another important family of PDEs, the so--called Temple class  \cite{MR1875343,Temple}.
\end{remark}

\subsection{Complete exceptionality and Monge--Amp\`ere invariants}
If $2^\textrm{nd}$ order PDEs are understood as hypersurfaces in $M^{(1)}$ (see \ref{subLagEq}), then the  complete exceptional ones form a  $\textrm{Cont}(M)$--invariant sub--class, where $\textrm{Cont}(M)$ is the set of the contactmorphisms of $M$. In particular, fibre--by--fibre, such PDEs must correspond to $\mathrm{CSp}_4$--invariant hypersurfaces in the three--dimensional Lagrangian Grassmannian $\LL(2,4)$, known also as the Lie quadric $Q^3$ \cite{2014arXiv1405.5198J} (recall that we set $n=2$) .  The theory of surfaces in $Q^3$ is thoroughly described,  e.g.,  in D. The's paper \cite{MR2876965}, so that the obvious question arise, how completely exceptional PDEs fit into the classification of surfaces in $Q^3$.\par
The answer, provided in this section, is that (for hyperbolic PDEs), the two equations \eqref{eqComPExpectConSegnoSbagliato} (one for each root of the symbol) correspond precisely to the vanishing of the Monge--Amp\`ere invariants $I_1$ and $I_2$ (see the above cited paper, Fig. 1). Incidentally, this provides a geometric proof to Proposition \ref{eq.Boillat.MAE}.\par
Recalling (see \ref{secProlong}) that to any point $m^1$ of a PDE $\E\subset M^{(1)}$ one can associate two (rank--one) directions in $T_{m^1}\E$ (or, better to say, in its vertical part), to which correspond two directions in $\CC_m$ (recall again that $n=2$). Thus, we have two sections
\begin{equation}\label{eqDefEllConI}
m^1\in\E\longmapsto \ell^i_{m^1}\in\p L_{m^1}\,,\quad i\in\{1,2\}
\end{equation}
of the projective line bundle
\begin{equation}\label{eqProjLinBund}
\p L\longrightarrow \E\, ,
\end{equation}
which is nothing but the projectivised tautological bundle restricted to $\E$.\par

The point $ \ell^i_{m^1}\in\p L_{m^1}$  determines the rank--one curve $(\ell^i_{m^1})^{(1)}\subset \E_m$, whose tangent direction is
$
\tilde{\ell}^i_{m^1}=T_{m^1}(\ell^i_{m^1})^{(1)}
$. 
 If $\ell^i_{m^1}=[\lambda^i(m^1):-1]$, then (see also Proposition \ref{propComeEFattoIlRaggio})
\begin{equation}\label{eqCampoVerticaleScrittaPerLEnnesimaVolta}
\tilde{\ell}^i_{m^1}   =  \left\langle\partial_{p_{11}}+\lambda^i(m^1) \partial_{p_{12}} +(\lambda^i(m^1))^2\partial_{p_{2}}\right\rangle\, .
\end{equation}

\begin{proposition}\label{propSecondaEquivalenzaCompEcc}
The derivative of the direction  \eqref{eqCampoVerticaleScrittaPerLEnnesimaVolta} along any curve   $\gamma$ passing through the point $m^1$ with tangent space $\tilde{\ell}^i_{m^1}$ is zero if and only  if \eqref{eqComPExpectConSegnoSbagliato} is satisfied in $m^1$.
\end{proposition}

\begin{proof}
 If $m^1(t)$ denotes the point moving along the curve $\gamma$, then $\ell^i_{m^1(t)}$, defined as in \eqref{eqDefEllConI}, is a point in $\p L_{m^1(t)}$.  So, in principle, the curve $\R\ni t\longmapsto \ell^i_{m^1(t)}\in \p L_{m^1(t)}$ takes values in different spaces. Nevertheless, for small values of $t$ each $L_{m^1(t)}$ can be identified with $L_{m^1}$, by means of the projection of $\CC_m=L_{m^1}\oplus V_m$ onto $L_{m^1}$, where $V_m=\Span{\partial_{p_1},\partial_{p_2}}$ is  the vertical subspace.\footnote{Observe that the theorem does not suffer from the choice of $V_m$: indeed, upon a different choice of $V_m$, the vector $(\lambda,-1)$ is replaced by an its linear image, and the vanishing of its derivative give the same condition on $\lambda$.}\par
 Since $\ell^i_{m^1(t)}=[\lambda^i(m^1(t)):-1]$ in the basis of $L_{m^1(t)}$, the projected line into $L_{m^1}$ will have the same components $[\lambda^i(m^1(t)):-1]$, but now in the basis of $L_{m^1}$. In other words, the curve $\ell^i_{m^1(t)}$ can be regarded as the curve $\R\ni t\longmapsto [\lambda^i(m^1(t)):-1]\in \p L_{m^1}$. In turn, the latter is equivalent to the curve $\R\ni t\longmapsto \lambda^i(m^1(t)) \in \R$, whose velocity at zero is precisely the left--hand side of \eqref{eqComPExpectConSegnoSbagliato}.
\end{proof}

Now we are ready to recast Proposition \ref{eq.Boillat.MAE} in a purely geometric perspective.

\begin{corollary}\label{corCorollarioso}
A hyperbolic $2^\textrm{nd}$ order PDE $\E$ is completely exceptional  if and only if it is   doubly ruled.
\end{corollary}
\begin{proof}
First of all, observe that  a hyperbolic $2^\textrm{nd}$ order PDE $\E$ is equipped with two distinct rank--one distributions, to which there are associated two distinct sections $\ell^1$ and $\ell^2$ of \eqref{eqProjLinBund}.\par

If $\E$ is completely exceptional, then the curves $(\ell^i_{m^1})^{(1)}$ are precisely the integral curves of the aforementioned rank--one distributions. Since, via Pl\"ucker embedding (see \eqref{eq.Plucker}), these $(\ell^i_{m^1})^{(1)}$ are lines, the surface $\E$ turns out to be doubly ruled (see also Example \ref{esFinale}).\par
Conversely, if $\E$ is doubly ruled, then $\E$ is described by two transversal families of curves $(\ell^i_{m^1})^{(1)}$, $i\in\{1,2\}$. Let us fix $i$ and consider, at each point $m^1\in\E$, the curve $\gamma^i:=(\ell^i_{m^1})^{(1)}$, whose coordinate representation is \eqref{eqParamCurvVert}. Since the section $m^1\to\ell^i_{m^1}$ is constant along $\gamma^i$, its derivative is zero, and hence, recalling Proposition \ref{propPrimaEquivalenzaCompEcc}, in view of Proposition \ref{propSecondaEquivalenzaCompEcc}, $\E$ is completely exceptional.
\end{proof}

\subsection{Geometric formulation of complete exceptionality: the \virg{symbolic way}}
Proposition \ref{propSecondaEquivalenzaCompEcc}
 gives a geometric meaning to the local formula \eqref{eqComPExpectConSegnoSbagliato}, i.e.,  to complete exceptional (hyperbolic) PDEs: it 
claims that  the roots of the polynomial \eqref{eq.symb.lambda} have zero acceleration along the rank--one directions determined (see Section \ref{secProlong}) by the roots themselves.\par  
 Nevertheless, it cannot be denied that hyperplane sections, i.e., Monge--Amp\`ere equations, exist also in the realm of non--hyperbolic equations. Consequently, there must be a natural way to extend the test of complete exceptionality to non--hyperbolic equations: the purpose of this section is to show that such an extension does exist, and it allows to characterise Monge--Amp\`ere equations in general.\par
%
The bottom line is simple: the ``zero acceleration condition'' for the roots of the characteristic polynomial \eqref{eq.symb.lambda}, used in Proposition \ref{propSecondaEquivalenzaCompEcc},  can be actually formulated without  solving the polynomial itself.
\begin{proposition}\label{prop.egregia}Let
$\mathcal{E}\subset M^{(1)}$  be a hyperbolic scalar second order PDE. Then $\E$ is completely exceptional iff $\mathcal{E}=\{F=0\}$ for some $F\in C^\infty(M^{(1)})$ such that $\Smbl^2(F)$ is proportional to $\Smbl(F)$.
\end{proposition}
\begin{proof}
Let $\ell^i=[\lambda^i:-1]$ be the roots of $\Smbl(F)$. This means that
\begin{equation}
[\Smbl(F)]=[ \mu^2-(\lambda^1+\lambda^2)\mu+\lambda^1\lambda^2]\, .
\end{equation}
Then, by  Remark \ref{remSimbCoord}, one has
\begin{align}
[ \Smbl ^2(F)]&=[\lambda^2\lambda^1_{p_{11}} +\lambda^1\lambda^2_{p_{11}}+(\lambda^2\lambda^1_{p_{12}}+\lambda^1\lambda^2_{p_{12}}-\lambda^1_{p_{11}}-\lambda^2_{p_{11}})\mu+\label{eqSimbSimb}\\
&+ (\lambda^2\lambda^1_{p_{22}}+\lambda^1\lambda^2_{p_{22}}-\lambda^1_{p_{12}}-\lambda^2_{p_{12}})\mu^2-(\lambda^1_{p_{22}}+\lambda^2_{p_{22}})\mu^3]\nonumber
\end{align}
Now, $\Smbl^2(F)$ is proportional to $\Smbl(F)$ if and only if an element $K+H\mu\in S^2L$ exists, such that
\begin{equation}\label{eqMaledetta}
[ \Smbl^2 (F)]=[ \Smbl (F) \cdot (K+H\mu)]\, .
\end{equation}
Since \eqref{eqMaledetta} is a polynomial identity of degree 3 in $\mu$, it gives 4 linear equations in $H$, $K$. Straightforward computations show that this system is compatible if and only if  the two equations, obtained by replacing $\lambda$ with $\lambda^i$ in \eqref{eqComPExpectConSegnoSbagliato}, are satisfied.
\end{proof}

\begin{corollary}\label{corCorollarioDelCacchio}
The equation $\E$ is A Monge--Amp\`ere equation if and only if, for some $F$ such that $\E=\{F=0\}$, $\Smbl^2(F)$ is proportional to $\Smbl(F)$. 
\end{corollary}
\begin{proof}
If $\E$ is hyperbolic, the result  immediately follows from Propositions \ref{eq.Boillat.MAE} and \ref{prop.egregia}. For a general $\E$, one must rely on the original direct integration method\footnote{This method   amounts at integrating \eqref{eq.system.Pi0S.2} without the hyperbolicity   condition.} of Boillat  \cite{MR1139843}.
\end{proof}
Now we can recast \eqref{eqComPExpectConSegnoSbagliato} in an intrinsic geometric way, not requiring the actual existence of the roots of the characteristic polynomial.
\begin{corollary}
Let us consider a local trivilization of line bundle \eqref{eqProjLinBund}
so that (w.r.t. this trivialization) each $\ell^i$ (see \eqref{eqDefEllConI}) can be seen as a function on $\E$.
Then hypersurface $\E$ is completely exceptional iff
\begin{equation}\label{condEgregiaEQ}
\Smbl(\ell^i)|_{T(\ell^i)^{(1)}}\equiv 0\, ,\quad i=1,2\,.
\end{equation}
\end{corollary}
\begin{proof}
Let $\ell=[\lambda:-1]$ be either $\ell^1$ or $\ell^2$, and just recall (see Proposition \ref{propComeEFattoIlRaggio}) that $T(\ell)^{(1)}$ is spanned by $\partial_{p_{11}}+\lambda\partial_{p_{12}}+\lambda^2\partial_{p_{22}}$.
Since locally $\ell\equiv \lambda$, one can compute the symbol of the function $\lambda$, as in \eqref{eqSmbl}, viz.
\begin{equation}\label{eqSecGradRadicConst2}
 \Smbl (\lambda)=\lambda_{p_{11}} \xi^2+\lambda_{p_{12}} \xi\eta+\lambda_{p_{22}} \eta^2\,  ,
 \end{equation}
 and pair it with $T(\ell)^{(1)}$. Hence,
   condition \eqref{condEgregiaEQ} coincides with \eqref{eqComPExpectConSegnoSbagliato}.
\end{proof}

\section{Conformal Geometry of 2D and 3D completely exceptional $2^\textrm{nd}$ order PDEs}\label{secCentrale}

Let $\Sigma$ be a hypersurface of a semi-Riemannian manifold $(M,g)$. We denote by $\mathrm{I}^g_\Sigma$ the first fundamental form of the hypersurface $\Sigma$ w.r.t. the metric $g$. Note that the tensor $\mathrm{I}^g_\Sigma$ can be degenerate; following the definitions of \cite{9780125267403}, we say that $\Sigma$ is a \emph{semi-Riemannian hypersurface} of $M$ if $\mathrm{I}^g_\Sigma$ is not degenerate. We
define the second fundamental form $\II^g_\eta$ of the hypersurface $\Sigma$ (w.r.t. a normal vector $\eta\in T^\perp_p\Sigma$) by
\begin{equation}\label{eq.second}
\II^g_\eta (v,w) = g (\nabla_vN,w )=-g(\nabla_XY,N)(p)\,, \quad v,w\in T_p\Sigma
\end{equation}
where $X$, $Y$ and $N$ are, respectively, local extension to $M$ of $v$, $w$ and $\eta$. Some consideration of definition \eqref{eq.second} is in order. Note that if $\Sigma$ is a  semi-Riemannian hypersurface, then $T_pM=T_p\Sigma\oplus T_p^\perp \Sigma$, $p\in \Sigma$. Here we use definition \eqref{eq.second} also when $\Sigma$ is not a semi-Riemannian hypersurface, i.e. when $T^\perp_p\Sigma\subset T_p\Sigma$. Since in our reasonings the orthogonal vector $\eta$ does not play an essential role, we will often use the symbol $\II^g_\Sigma$ to denote the second fundamental form of $\Sigma$.

\medskip
Let us denote by $\mathrm{sym}_{\mathrm{max}}(T_n)$ the set of infinitesimal conformal symmetries of the tensor $T_n$ of maximal rank. If $n=2$, then we put $\mathrm{sym}_{\mathrm{max}}(T_2)=C^\infty(\L(2,4))$. Note that if $X$ is a vector fiend on $\L(3,6)$ of maximal rank then $X\lrcorner T_3$ is a non-degenerate symmetric $2$-tensor. We recall that for hyperplane sections of $\L(n,2n)$ we mean hyperplane sections of $\L(n,2n)$ via the Pl\"ucker embedding \eqref{eq.Plucker}.
\begin{theorem}\label{th.main.1}
Let $\Sigma$ be a hypersurface of $\L(n,2n)$, $2\leq n\leq 3$ and $X_{n-2}\in \mathrm{sym}_{\mathrm{max}}(T_n)$. Then
the condition
\begin{equation}\label{cond.1}
\II_\Sigma^{X_{n-2}\lrcorner T_n}\in (S^2_{T_n})|_\Sigma\,,
\end{equation}
is independent of $X_{n-2}$. Furthermore, $\Sigma$ satisfies condition \eqref{cond.1} iff it is a hyperplane section.
\end{theorem}

\begin{corollary}\label{cor.main.1}
Let $\Sigma$ be a semi-Riemannian hypersurface of $(\L(2,4),T_2)$.
The following statements are equivalent:
\begin{enumerate}
\item $\Sigma$ is a hyperplane section;
\item $\II_\Sigma^{T_2}$ is proportional to $\I_\Sigma^{T_2}$;
\item the trace-free part of second fundamental form of $\Sigma$ is zero.
\end{enumerate}
\end{corollary}

In the next sections \ref{sec.hyper.sec.2} and \ref{secL36} we prove Theorem \ref{th.main.1} and Corollary \ref{cor.main.1} in the case $n=2$ and $n=3$, respectively.

\subsection{Hyperplane sections of $\L(2,4)$}\label{sec.hyper.sec.2}

Since on $\L(2,4)$ there exists a canonical conformal semi-Riemannian metric $T_2$, before proving Theorem \ref{th.main.1} and Corollary \ref{cor.main.1} we will see how the main object we use transforms under a conformal changing of $T_2$.

\smallskip
Let $(M,g)$ be a (semi-)Riemannian manifold.
Let $\tilde{g}=e^{2\lambda}g$, where $\lambda\in C^\infty(M)$. Let $\nabla^{\tilde g}$ be the Levi-Civita connection of the metric $\tilde{g}$. Then
\begin{equation}\label{eq.conn.up.to.conf}
\nabla^{\tilde g}_XY=\nabla^g_XY + \beta(X,Y)
\end{equation}
where $\beta:TM\times TM\to TM$ is a tensor defined by
\begin{equation}\label{eq.beta}
\beta(X,Y) = X(\lambda)Y + Y(\lambda)X - g(X,Y)g^{-1}(d\lambda)
\end{equation}
Let us see how connection $\nabla^{\tilde g}$ acts on $k$-forms. Let $\omega\in \Lambda^kM$. Then
\begin{equation}\label{eq.dual.conn.conf.1}
\nabla^{\tilde g}_X(\omega)(Y_1,\dots,Y_k)={\nabla}^g_X(\omega)(Y_1,\dots,Y_k)-\omega(\beta(X,Y_1),Y_2,\dots,Y_k) -\dots-\omega(Y_1,\dots,\beta(X,Y_k))
\end{equation}
By using formula \eqref{eq.dual.conn.conf.1}, it is easy to realize that
\begin{equation}
H^{\tilde g}(f) = \nabla^{\tilde g}(df) = \nabla^g(df) - df\circ\beta = H^g(f) - 2 d\lambda\odot df + g^{-1}(d\lambda,df)g
\end{equation}
In view of formula \eqref{eq.second},  \eqref{eq.conn.up.to.conf} and \eqref{eq.beta} we have that
\begin{equation}\label{eq.second.conf}
{\II}^{\tilde g}_\eta=e^{2\lambda}(\II^g_\eta + \eta(\lambda)g)
\end{equation}
where $\lambda$ is understood restricted on $\Sigma$.
\begin{remark}
In the case we choose a \emph{unit} vector field $N$, the right transformation formula for ${\II}^{\tilde g}_\eta$, is
\begin{equation}\label{eq.second.conf.wong}
{\II}^{\tilde g}_\eta  = e^{\lambda}\II^g_\eta + \eta(e^{\lambda})g
\end{equation}
\end{remark}
Let $\Sigma$ be a hypersurface of a semi-Riemannian manfold. Let $\eta\in T^\perp \Sigma$. Let us consider the trace-free part $\II^{g\,0}_\eta$ of the second fundamental form of $\Sigma$:
\begin{equation}\label{eq.trace.free.second}
\II^{g\,0}_\eta=\II^g_\eta-\frac{1}{n-1}H_\eta\cdot g
\end{equation}
where $H_\eta=g^{ij}\mathrm{II_\eta}_{ij}$
is the mean curvature vector.
\begin{remark}
$\II^{g\,0}_\eta$ is conformally invariant, whereas the $(1,1)$-tensor $g^{ik}{\II^{g\,0}_\eta}_{kj}$ is invariant up to conformal transformations. Note that the zero set of both $\II^g_\eta$ and $\II^{g\,0}_\eta$ of $\Sigma$ are independent of the normal field $\eta$.
\end{remark}
As for $\II^g_\eta$, in our reasonings the orthogonal vector $\eta$ in $\II^{g\,0}_\eta$ does not play an essential role, as we are mainly interested in the zero set of $\II^{g\,0}_\eta$. So, we will often use the symbol $\II^{g\,0}_\Sigma$ to denote the trace-free part of the second fundamental form of $\Sigma$.\par

Next propositions prove Theorem \ref{th.main.1} and Corollary \ref{cor.main.1} in the case $n=2$.

\begin{proposition}\label{prop.hyperplane}
A hypersurface $\Sigma$ of $(\L(2,4),T_2)$ is a hyperplane section of $\L(2,4)$ if and only if $\II^{T_2}_\Sigma$ is proportional to $\I^{T_2}_\Sigma$. Furthermore, if $\Sigma$ is a semi-Riemannian hypersurface $\Sigma$ of $(\L(2,4),T_2)$, then it is a hyperplane section if and only if $\II^{T_2\,0}_\Sigma=0$.
\end{proposition}
\begin{proof}
Without loss of generality, we suppose that $\Sigma$ is locally described by $\Sigma=\{F=p_{22}-h(p_{11},p_{12})=0\}$
so that $T\Sigma=\langle (1,0,h_{p_{11}}),(0,1,h_{p_{12}}) \rangle$. The second fundamental form $\II^{T_2}_\Sigma$ is proportional to the first fundamental form $\I^{T_2}_\Sigma$ (here we do not assume that $\I^{T_2}_\Sigma$ is necessarily non-degenerate) if and only if system \eqref{eq.system.Pi0S.2} is satisfied (we recall that here $h=h(p_{11},p_{12})$). As we already seen, this system can be explicitly solved: its general solution $h$ gives a Monge-Amp\`ere equation $\{F=p_{22}-h=0\}$ \cite{MR1139843,MR0481580}.

\smallskip

Let now assume that $\Sigma$ is semi-Riemannian, so that the determinant of the first fundamental form $\I^{T_2}_\Sigma\neq 0$ is not zero, i.e. $4h_{p_{11}}+h_{p_{12}}^2\neq 0$. The condition $\II^{T_2\,0}_\Sigma=0$ gives the following system of three equations:
\begin{equation}\label{eq.system.Pi0S}
\left\{
\begin{array}{l}
2h_{p_{11}p_{11}}h_{p_{11}} + h_{p_{11}p_{11}}h_{p_{12}}^2 -2h_{p_{11}}h_{p_{12}}h_{p_{11}p_{12}} + 2h_{p_{11}}^2h_{p_{12}p_{12}}=0
\\
\\
4h_{p_{11}p_{12}}h_{p_{11}} - h_{p_{12}} h_{p_{11}p_{11}} + h_{p_{12}} h_{p_{11}} h_{p_{12}p_{12}}=0
\\
\\
2h_{p_{11}} h_{p_{12}p_{12}} + h_{p_{12}p_{12}} h_{p_{12}}^2 + 2h_{p_{11}p_{11}} + 2h_{p_{12}}h_{p_{11}p_{12}} =0
\end{array}
\right.
\end{equation}
If $h_{p_{11}}=0$, then the first two equations of \eqref{eq.system.Pi0S} are identically satisfied, whereas the third one gives $h_{p_{12}p_{12}}=0$, so that $h=K_1p_{12}+K_2$. The same reasoning holds if $h_{p_{12}}=0$, so that in the remaining part of the proof we suppose
\begin{equation}\label{eq.hr.ns.not.0}
h_{p_{11}}\neq 0\,, \quad h_{p_{12}}\neq 0
\end{equation}
System \eqref{eq.system.Pi0S} is equivalent to \eqref{eq.system.Pi0S.2}.
In fact, if we denote by $\mathrm{eq_1}$ and $\mathrm{eq_2}$, respectively, the first and the second left-hand side term of system \eqref{eq.system.Pi0S.2}, then system \eqref{eq.system.Pi0S} reads as follows:
\begin{equation}\label{eq.system.Pi0S.3}
\left\{
\begin{array}{l}
(2h_{p_{11}}+h_{p_{12}}^2)\mathrm{eq_1} - h_{p_{11}}h_{p_{12}}\mathrm{eq_2}=0
\\
\\
-h_{p_{12}}\mathrm{eq_1}+2h_{p_{11}}\mathrm{eq_2}=0
\\
\\
2\mathrm{eq_1}+h_{p_{12}}\mathrm{eq_2}=0
\end{array}
\right.
\end{equation}
The solution of this system is $\mathrm{eq_1}=\mathrm{eq_2}=0$ if and only if $h_{p_{11}}h_{p_{12}}(4h_{p_{11}}+h_{p_{12}}^2)\neq 0$. This inequality is satisfied as we are supposing \eqref{eq.hr.ns.not.0} and that $\Sigma$ is a semi-Riemannian hypersurface.
\end{proof}

\begin{proposition}
Let $\Sigma$ be a hyperplane section of $\L(2,4)$. Then it corresponds to a hyperbolic (resp. elliptic, parabolic) Monge-Amp\`ere equation iff $\det(\I^{T_2}_\Sigma)<0$ (resp. $\det(\I^{T_2}_\Sigma)>0$, $\det(\I^{T_2}_\Sigma)=0$).
\end{proposition}
\begin{proof}
We recall that a Monge-Amp\`ere equation
$$
k_0(p_{11}p_{22}-p_{12}^2) + k_1 p_{11} + k_2 p_{12} + k_3 p_{22} + k_4=0
$$
is hyperbolic (resp. elliptic, parabolic) iff its discriminant
$$
\Delta=k_2^2-4k_1k_3+4k_0k_4
$$
is greater than (resp. less than, equal to) zero. By using the notation of Proposition \ref{prop.hyperplane}, we have that
$$
\det(\I^{T_2}_\Sigma)=-\frac{1}{4}\frac{\Delta}{(k_3+k_0p_{11})^2}
$$
and the proposition follows.
\end{proof}

\subsection{Hyperplane sections of $\L(3,6)$}\label{secL36}


Infinitesimal conformal symmetries of $T_3$ form a $21$--dimensional Lie algebra, which by dimensional reasons must coincide with $\Sp(6)$. Below we write a list of generators.
$$
\partial_{\xi_i}\,,
\quad \xi_i\partial_{\xi_\ell} + \frac{1}{2}\xi_j\partial_{\xi_m} + \frac{1}{2}\xi_k\partial_{\xi_n}\,,
\quad
\xi_i^2\partial_{\xi_1} + \xi_i\xi_j \partial_{\xi_2}   +\xi_i\xi_k\partial_{\xi_3} + \xi_j^2  \partial_{\xi_4}  +\xi_j\xi_k\partial_{\xi_5} + \xi_k^2  \partial_{\xi_6}
$$
$$
\xi_1\xi_2 \partial_{\xi_1} + \frac{1}{2}(\xi_2^2+\xi_1\xi_4) \partial_{\xi_2} + \frac{1}{2}(\xi_2\xi_3+\xi_1\xi_5) \partial_{\xi_3} + \xi_4\xi_2\partial_{\xi_4} + \frac{1}{2}(\xi_5\xi_2+\xi_3\xi_4) \partial_{\xi_5} + \xi_3\xi_5 \partial_{\xi_6}
$$
$$
\xi_1\xi_3 \partial_{\xi_1} + \frac{1}{2}(\xi_2\xi_3+\xi_1\xi_5) \partial_{\xi_2} +  \frac{1}{2}(\xi_3^2+\xi_1\xi_6) \partial_{\xi_3} +  \xi_5\xi_2 \partial_{\xi_4} +  \frac{1}{2}(\xi_3\xi_5+\xi_2\xi_6) \partial_{\xi_5} +  \xi_6\xi_3\partial_{\xi_6}
$$
$$
\xi_2\xi_3 \partial_{\xi_1} + \frac{1}{2}(\xi_5\xi_2+\xi_3\xi_4) \partial_{\xi_2} + \frac{1}{2}(\xi_3\xi_5+\xi_2\xi_6) \partial_{\xi_3} + \xi_4\xi_5 \partial_{\xi_4} + \frac{1}{2}(\xi_5^2+\xi_4\xi_6) \partial_{\xi_5} + \xi_6\xi_5 \partial_{\xi_6}
$$
$(i,j,k),(\ell,m,n)\in\{(1,2,3),(2,4,5),(3,5,6)\}$
where we used the notation $(\xi_1,\xi_2,\dots,\xi_6)=(p_{11},p_{12},\dots,p_{33})$

\smallskip
We note that the above generators are all of non-maximal rank. Nevertheless the stretching $\sum_{i\leq j}p_{ij}\partial_{p_{ij}}$, which can be obtained as a particular linear combination with constant coefficients of the above vector fields, is of maximal rank. A straightforward computation based also on the usage of a Maple symbolic computation package, shows that condition \eqref{cond.1}, for $n=3$, is independent of the chosen infinitesimal symmetry of maximal rank. If $\E=\{F=0\}$, condition \eqref{cond.1} is described by
$
F_{p_{ij}p_{hk}}\xi_i\xi_j\xi_h\xi_k=G^{{ij}}F_{p_{hk}}\xi_i\xi_j\xi_h\xi_k
$, for some $G^{ij}$, which means precisely that the symmetric derivative of the symbol is proportional to the symbol itself.
The result thus follows\footnote{Recall  that  Corollary \ref{corCorollarioDelCacchio} was formulated in the hypothesis $n=2$, though its validity is general (see Boillat  \cite{MR1139843}).} from Corollary \ref{corCorollarioDelCacchio}.


\section{Higher degree hypersurface sections of $\LL(n,2n)$}\label{sec.Jan}


The main result of  Section \ref{secCentrale} above is Theorem \ref{th.main.1}. It allows to decide when a hypersurface $\Sigma$ in the Lagrangian Grassmannian   $ X_n = \LL(n,2n) $ is a hyperplane section, by using tensors naturally associated with $\Sigma$.  Doing so, Theorem \ref{th.main.1} revealed  the intrinsic character of  complete exceptionality, and this is its main strength. The drawback is that Theorem \ref{th.main.1} has been proved only for $n=2,3$, and it  still has to be extended   to the case of higher--degree sections.\par
In this section, we embrace an extrinsic perspective, which will lead us  to characterise certain distinguished functions on $X_n$, cutting out   $r^\textrm{th}$ degree hypersurface sections $\Sigma$   of $X_n$, \emph{for arbitrary values of $n$ and $r$}. To this end,  we shall exploit the natural action of the symplectic group
$\Sp_{2n}$ on    $ X_n$, 
and we shall find a natural differential operator between
certain bundles over $X_n$, such that the zero--loci of the sections in its kernel
are precisely the (regular) intersections with  $r^\textrm{th}$ degree  hypersurfaces in the
ambient projective space $\p^M$.\par
 The main result of this section, Theorem \ref{thm:final-bgg}, is thus more general than      Theorem \ref{th.main.1}, since it is valid for all dimensions $n$ and degrees $r$, but, at the same time, it is less direct: 
rather than characterising a hypersurface $\Sigma \subset X_n$ in intrinsic terms, it only
characterises certain distinguished functions cutting $\Sigma$ out. \par
Needless to say, Theorem \ref{th.main.1}  and Theorem \ref{thm:final-bgg} match on their common overlapping. In Section \ref{subsecLinkSecJanConLaPrecedente} we show how to obtain from Theorem \ref{thm:final-bgg} an intrinsic \emph{necessary} condition for complete exceptionality which, expectedly, turns out to be also sufficient in the cases $n=2,3$ covered by Theorem \ref{th.main.1}.\par
The reader already familiar with the BGG  techniques can safely jump to the main result, Theorem \ref{thm:final-bgg}, by skipping the subsections   \ref{eqSubSecJan1}--\ref{eqSubSecJanUltima}, whose purpose is that of providing a minimal self--consistent  background on the subject for the  non--experts. Apparently, there is no  alternative way to prove the main  Theorem \ref{thm:final-bgg}, other than the BGG resolution, in spite of the simple formulation of the Theorem itself.   \par
Following the proof of Theorem \ref{thm:final-bgg}, there are   two concluding sections. 
 In Section \ref{secLinkWithProjGeom}  we recall that $X_n$ is rational, and use this fact to obtain a simple confirmation of the fact that    the  functions cutting out completely exceptional PDEs are    suitable pull--backs of the linear functions on $\p^M$. 
 In   Section \ref{secLinkWithJets} we show that the $r^\textrm{th}$ jet space of hypersurfaces of $X_n$ is foliated by canonical equations, each of which is (noncanonically) equivalent to the equation of (generalised) complete exceptionality, thus confirming the natural character of this property.

%

\subsection{The homogeneous structure of Lagrangian Grassmannian and its minimal projective embedding}\label{eqSubSecJan1}
Our point of departure is the fact that $X_n$ is a homogeneous space for
$\Sp_{2n}$. In order to make it explicit, we identify $X_n$ with the Grassmannian of
$n$--dimensional Lagrangian subspaces of $\R^n \oplus \R^{n*}$, where the
latter space carries the tautological symplectic structure $\omega$  defined by
$\omega|_{\R^n}\equiv 0$, $\omega(v,\alpha):=\langle v,\alpha \rangle$ for all $v\in\R^n$ and $\alpha\in\R^{n\ast}$, and  $\omega|_{\R^{n\ast}}\equiv 0$. 
The group $\Sp_{2n}$ is precisely the automorphism group of $\omega$, thus obviously
acting on $X_n$. We fix an origin in $X_n$, corresponding to the isotropic subspace
$\R^{n*}$, and we let $P \subset \Sp_{2n}$ be its stabiliser. The group $P$ is
a semi--direct product $P \simeq \GL_n \ltimes S^2\R^{n*}$.
Its Levi factor $\GL_n$
is the subgroup of $\Sp_{2n}$ preserving the direct sum decomposition
in $\R^n\oplus\R^{n*}$, i.e., stabilising the complementary $\R^n$ in addition to $\R^{n*}$.
The unipotent radical $S^2\R^{n*}$ is abelian and maps $\R^n$ into $\R^{n*}$ in the essentially
unique $\GL_n$-equivariant manner. In the following discussion we shall often
work with representations of both $\Sp_{2n}$ and the subgroup $\GL_n$; in particular,
an irreducible representation of $\Sp_{2n}$ will often be decomposed into $\GL_n$-irreducible
summands.

The Pl\"ucker embedding (introduced in Section \ref{sec.Plucker}) is $\Sp_{2n}$-equivariant and realises $X_n$
as a subvariety in the projectivisation of the space of $n$-forms
on $\R^n\oplus\R^{n*}$ orthogonal to $\omega$ (i.e., precisely  the space $\p^M$ introduced  in   Remark \ref{remMinimality}). This latter space,
temporarily denoted $W$,
is the $\Sp_{2n}$-irreducible summand of $\Lambda^n (\R^n\oplus\R^{n*})$
containing the one-dimensional subspace $\det\R^{n*}$. We then
embed
$ X_n \hookrightarrow \p W $
identifying the Lagrangian Grassmannian
with the $\Sp_{2n}$-orbit of the point corresponding to $\det\R^{n*}$.
Now, given a non--zero element $f \in S^r W^*$ for some $r > 0$, we may
intersect $X_n$ with the zero--locus of $f$: the resulting subset $\Sigma \subset X_n$
consists of points corresponding to one--dimensional subspaces of $W$ on which
$f$ restricts to zero. In general, it may happen that $\Sigma$ is all of $X_n$,
i.e., $X_n$ is contained in the zero--locus of $f$; as we shall see later, one
only needs to consider the $\Sp_{2n}$-irreducible summand of $S^r W^*$
containing the image of $W^*$ under the $r^\textrm{th}$  power map.
\subsection{The idea of  the BGG resolution of the space of $r^\textrm{th}$ degree sections}
Before proceeding to the more technical part of this section, let us
for a moment pretend that the elements of $S^r W^*$ give rise to
actual functions on $X_n$, rather than sections of a line bundle---for
instance, intersect with a standard affine open in $\p W$ and trivialise.
We would then obtain a finite dimensional subspace $F \subset C^\infty(X_n)$,
with the property that $r^\textrm{th}$ degree   hypersurface sections of $X_n$ are
precisely the zero--loci of non--trivial elements of $F$. How does one
decide whether a given smooth function $f \in C^\infty(X_n)$ belongs
to $F$ though? Solving for a linear combination of some basis functions
is certainly not a viable method. Instead, one seeks to identify $F$ as
the kernel of a differential operator. The basic example is
the identification of constant functions with the kernel of the
exterior derivative $d : C^\infty(X_n) \to \Omega^1(X_n)$. While trivial
from the point of view of our problem, it fits into an important
object---the de Rham resolution of the locally constant sheaf
$
0 \to \underline\R \to \Omega^0 \to \Omega^1 \to \dots
$
by differential operators. It turns out that our problem of characterising
$F$ as the kernel of a differential operator finds its place and solution
in an analogous resolution. The latter, known as the BGG resolution
(or rather its generalised version due to Lepowsky), resolves
a finite--dimensional irreducible representation of $\Sp_{2n}$ by
a sequence of equivariant differential operators between certain natural
vector bundles on $X_n$. Its first map is a differential operator
acting on sections of the line bundle $\mathcal{O}_{X_n}(r)$, the $r^\textrm{th}$ power
of the dual of the tautological line bundle of $\p W$ restricted to $X_n$. The kernel
of this operator consists precisely of restrictions to $X_n$ of elements of $S^r W^*$, which
  cut  out $r^\textrm{th}$ degree  hypersurface sections of $X_n$.
\subsection{Representation--theoretic preliminaries}
In order to define the bundles and operators involved, we need to set up
some rudimentary representation--theoretic notation. Passing to the level
of Lie algebras, we have the inclusions
$
\gl_n \subset \gp_n \subset \sp_{2n}$ and $
\gp_n \simeq \gl_n \oplus S^2\R^{n*}$  (semidirect sum).
In fact, $\gp_n$ admits a $\GL_n$-invariant complement, thus giving rise
to a $\GL_n$--equivariant decomposition
\begin{equation}\label{eqn:graded-decomp}
\sp_{2n} = S^2 \R^n \oplus \gl_n \oplus S^2 \R^{n*}\, ,
\end{equation}
where the extremal summands are abelian, and their commutator lies in the middle one.
Recall that, being a split real form of a simple Lie algebra, $\sp_{2n}$ contains
a unique up to conjugacy split Cartan subalgebra. That is, a maximal abelian
self--normalising subalgebra $\gh$ such that $\sp_{2n}$ decomposes as a direct sum
of $\gh$ and finitely many
one--dimensional $\gh$--invariant subspaces
labelled by distinct elements of $\gh^*$. These one-dimensional summands are the \emph{root
subspaces}. An element $H \in \gh$ acts on a root subspace labelled by $\alpha \in \gh^*$
as the scalar $\alpha(H)\in\R$. The subset of $\gh^*$ consisting of these labels
is the \emph{root system}, and its elements are the \emph{roots}.
To each root $\alpha$ one associates the \emph{coroot} $H_\alpha \in \gh$,
acting as $\pm 2$ on the root subspaces labelled by $\pm\alpha$ and,
together with these, spanning a sub--algebra isomorphic to $\mathfrak{sl}_2$.\par
Now, this decomposition
of the adjoint representation extends to an arbitrary finite--dimensional representation
of $\sp_{2n}$. Namely, a finite--dimensional representation decomposes under the
action of $\gh$ into a direct sum of invariant subspaces labelled by
distinct elements of $\gh^*$, this time called \emph{weights}: an element
$H \in \gh$ acts on a weight subspace with weight $\lambda\in\gh^*$ as the scalar
$\lambda(H) \in \R$. Thus, roots are simply the weights of the adjoint representation.
The root system is invariant under reflection about the origin $0 \in \gh^*$. One may
choose a hyperplane in $\gh^*$ such that half of the roots lie in one of the resulting
half--spaces, and declare these as the \emph{positive roots}. Given such choice,
one says that a non--zero element of a given representation of $\sp_{2n}$ is
a \emph{highest weight vector} if it is annihilated by all elements of $\sp_{2n}$
lying in root subspaces labelled by positive roots. In case of an \emph{irreducible}
finite--dimensional representation, a highest weight vector is indeed a
weight vector, i.e., an element of a weight subspace, and the weight that
labels it is called the \emph{highest weight}. Then, the highest weight
is unique, with a one--dimensional weight subspace, and characterises the
irreducible representation up to isomorphism. An element $\lambda\in\gh^*$
is the highest weight of an irreducible finite--dimensional representation
if and only if $\lambda(H_\alpha)$ is a non--negative integer
for each positive root $\alpha$.
Such weights
are called \emph{dominant integral}.

In our case, we will choose the Cartan subalgebra $\gh$ to consist of
the diagonal matrices in $\gl_n \subset \sp_{2n}$. Then, each
summand in \eqref{eqn:graded-decomp} decomposes into a
direct sum of root subspaces, plus $\gh$ in $\gl_n$. We choose
the positive roots so that they label root subspaces spanning
$S^2\R^{n*}$ plus the space of upper--triangular matrices in $\gl_n$.
Now, given a dominant integral weight $\lambda$ we have the
irreducible representation $V_\lambda$ with highest weight $\lambda$;
we choose a highest weight vector $v_\lambda \in V_\lambda$, unique up to scale.
Restricting $V_\lambda$ to the reductive subalgebra $\gl_n$ gives rise to
a decomposition into $\GL_n$--irreducibles; the unique one containing $v_\lambda$
will be denoted $L_\lambda$. Finally, we note that since $\gh$ is also a Cartan
subalgebra for $\gl_n$, we may use the elements of $\gh^*$ to label irreducible
finite--dimensional representations of $\gl_n$. Some of these are the $L_\lambda \subset V_\lambda$
as described above, but one may also make sense of $L_\lambda$ for $\lambda\in\gl_n$
that is dominant integral for $\gl_n$, but not for $\sp_{2n}$ (here we consider
the root system of $\gl_n$ as a subsystem of the root system of $\sp_{2n}$, with
a compatible choice of positive roots, whence the condition on
$\lambda(H_\alpha)$ is imposed only for a subset of positive roots $\alpha$).
\subsection{Tautological bundle and its powers as associated bundles}
We are now ready to describe the Pl\"ucker embedding $X_n \to \p W$ in our
newly set--up notation. First, $W$ is an irreducible representation of $\sp_{2n}$,
so we may write $W = V_\lambda$ for some dominant integral weight $\lambda$ (we shall
from now on fix $\lambda$ to denote this particular weight). Then,
an inspection of the action of $S^2\R^{n*}$ on $W$ shows that the elements of
$\det\R^{n*}$ are the highest weight vectors. Accordingly, $X_n \simeq \Sp_{2n}/P$
is the orbit of the highest weight line $[v_\lambda]$ in $\p V_\lambda$. But since
$\GL_n$ acts on $\det\R^{n*}$, we actually have $[v_\lambda] = L_\lambda$. Now,
given $g \in \Sp_{2n}$, we may view $gL_\lambda \subset V_\lambda$
as the fibre of the \emph{tautological
line bundle} at $gP \in X_n$. The action $\Sp_{2n} \times L_{\lambda} \to V_\lambda$
thus allows us to identify the tautological line bundle over $X_n$ with
the associated vector bundle
 $
\LLL_\lambda := \Sp_{2n} \times^P L_\lambda
 $,
where $S^2\R^{n*} \subset P$ acts on $L_\lambda$ trivially by definition.
We may in fact apply this construction for any integral dominant weight
$\mu$: first extend the $\GL_n$--irreducible representation $L_\mu$
to a representation of $P$ with a trivial action of $S^2\R^{n*}$, then
form the associated bundle $\LLL_\mu$. Note that the $\Sp_{2n}$--action
on $X_n$ lifts to the total space of $L_\mu$, and thus turns its
space of global sections in to a representation.\par
But keep in mind our main goal:
we are interested in $r^\textrm{th}$ degree  hypersurface sections of $X_n$, and
thus in elements of $S^r V_\lambda^*$. The symmetric power $S^r V_\lambda$ is
in general not irreducible; there is however a unique irreducible summand containing
$v_\lambda^r$, and as such necessarily isomorphic to $V_{r\lambda}$. Furthermore, there
is no other summand with the same highest weight; accordingly,
the dual $S^r V_\lambda^*$ contains a unique irreducible summand isomorphic to $V_{r\lambda}^*$.
Once again, the line spanned by $v_\lambda^r$ is actually $\GL_n$--invariant, so that
$L_{r\lambda} = \langle v_\lambda^r\rangle \simeq L_\lambda^r$. Observe that evaluating
an element $f \in S^r V_\lambda^*$ on the tautological bundle of $X_n$ gives rise to a
global section of $S^r \LLL_\lambda^* \simeq \LLL_{r\lambda}^*$. By the associated bundle
construction, the latter may be identified with a $P$--equivariant map
$\Sp_{2n} \to L_{r\lambda}^*$.
We may then conveniently state the following.
\begin{lemma}\label{lem:eval-action}
There is a commutative diagram of $\Sp_{2n}$--equivariant maps
\begin{equation*}\xymatrix{
**[l]S^r V_\lambda^* \ar[r] \ar[d]_{\textrm{\normalfont pr}} &**[r] \Gamma(X_n, \LLL_{r\lambda}^*)   \ar[d]^{\rotatebox{-90}{$\simeq$}} \\
**[l]V_{r\lambda}^* \ar[r]&**[r] C^\infty(\Sp_{2n}, L_{r\lambda}^*)^P
}\end{equation*}
where the top horizontal arrow is the evaluation map, while the
bottom one is adjoint to the action map
$\Sp_{2n} \times L_{r\lambda} \to V_{r\lambda}$.
\end{lemma}
The proof is a left as a simple yet instructive exercise for the reader. Note that
we do in particular find that the remaining irreducible summands in $S^r V_\lambda^*$
consist of  $r^\textrm{th}$ degree homogeneous polynomials vanishing identically on $X_n$.
Hence, the actual $r^\textrm{th}$ degree   hypersurface sections of $X_n$ are precisely the
zero--loci of non--zero elements in the image of
$V_{r\lambda}^* \hookrightarrow \Gamma(X_n, \LLL_{r\lambda}^*)$.
Our task is thus to express this image as the kernel of a differential
operator. The BGG machinery will do this as a first step in the resolution
of $V_{r\lambda}^*$ by equivariant differential operators between natural vector bundles
of the form $\LLL_\mu$. The easiest way to introduce it is in a dual picture, to which
we shall now pass.
\subsection{The dual picture via Verma modules}
Recall how passing from $L_\mu^*$ to $\Gamma(X_n, \LLL_\mu^*)$
produced a representation of $\Sp_{2n}$ from a representation of $\GL_n$.
There is a dual algebraic variant of this construction, producing
a representation of $\sp_{2n}$. We will use $U(-)$ to denote the universal
enveloping algebra functor from Lie algebras to associative algebras over $\R$.
It is customary to speak of $U(\sp_{2n})$--modules rather than representations
of $\sp_{2n}$: these are equivalent. Given a weight $\mu$, dominant integral
for $\gl_n$, we view $L_\mu$ as a $U(\gp)$--module with a trivial action of
$S^2\R^{n*}$. Then, we form the so--called \emph{Verma module} $ U(\sp_{2n}) \otimes_{U(\gp)} L_\mu. $ with highest weight $\mu$, which   is a $U(\sp_{2n})$--module by left action on the first factor. The
induced structure of a $U(\gp)$--module coincides with the natural one
on each factor, and in fact integrates to a representation of $P$. This allows us
to form an associated bundle and leads to a geometric interpretation.
\begin{proposition}\label{pro:verma-jet}
There is a natural $\Sp_{2n}$-equivariant
identification
\begin{equation}
\label{eq:verma-jet}
\Sp_{2n} \times^P [ U(\sp_{2n}) \otimes_{U(\gp)} L_\mu ]^* \simeq J^\infty \LLL_\mu^*
\end{equation}
with the bundle of infinite jets of sections of $\LLL_\mu^*$ over $X_n$.
\end{proposition}
\begin{proof}[Sketch of a proof]
The basic ingredients of the proof are: (1) identify sections
of an associated bundle with $P$--equivariant maps from $\Sp_{2n}$, (2)
represent $\sp_{2n}$ as vector fields on $\Sp_{2n}$, (3)
represent $U(\sp_{2n})$ as differential operators on $\Sp_{2n}$.
\end{proof}

We thus see that Verma modules are, in a suitable sense, dual to infinite jet bundles.\footnote{Since $U(\sp_{2n})\otimes_{U(\gp)}L_\mu$ is not just
a representation of $P$ but also of $\sp_{2n}$, the left-invariant Maurer--Cartan form
on $\Sp_{2n}$ induces a flat connection on the associated vector bundle:  the canonical flat connection on the bundle of infinite jets.}
Recalling the relation between jets and differential operators, we further identify
homomorphisms of Verma modules with equivariant differential operators:
\begin{corollary}\label{cor:equiv-diff}
Let $\mu, \nu \in \gh^*$ be a pair of dominant integral weights
for $\gl_n$. There is a natural isomorphism
$
\Hom_{U(\sp_{2n})}(
U(\sp_{2n})\otimes_{U(\gp)} L_\mu,
U(\sp_{2n})\otimes_{U(\gp)} L_\nu) \simeq
\mathrm{Diff}(\LLL_\nu^*, \LLL_\mu^*)^{\Sp_{2n}}
$
between the space of $U(\sp_{2n})$-module homomorphisms
on the left hand side and the space of $\Sp_{2n}$--equivariant
differential operators on the right hand side.
\end{corollary}
\begin{proof}
A differential operator $D$ from $\LLL_\nu^*$ to $L_\mu^*$ is a vector bundle map
$J^\infty \LLL_\nu^* \to \LLL_\mu^*$ factoring through some finite jet bundle. An inspection
of the identification of Proposition  \ref{pro:verma-jet} shows that the cofiltration
on $J^\infty L_\nu^*$ induced by finite jet projections
$J^\infty\LLL_\nu^* \to J^\ell\LLL_\nu^*$ coincides
with the cofiltration on the left hand side of \eqref{eq:verma-jet} induced by
the standard filtration $U^{\le \ell}(\sp_{2n})$ on the universal enveloping algebra.
Hence, the operator $D$ is the same as a $P$--equivariant function
on $\Sp_{2n}$ with values in $\Hom_\R(L_\mu, U(\sp_{2n})\otimes_{U(\gp)} L_\nu)$. It is
$\Sp_{2n}$--equivariant if and only if
the function is constant and its value is a $U(\gp)$--module
homomorphism $L_\mu \to U(\sp_{2n})\otimes_{U(\gp)} L_\nu$. Equivalently, by the universal
property of $U(-)$, an equivariant operator $D$ is the same as
a $U(\sp_{2n})$-module homomorphism of the Verma modules associated with $L_\mu$ and $L_\nu$.
\end{proof}

Before we use this result to look for a differential operator resolving $V_{r\lambda}^*$,
let us pin--point one last element: the dual algebraic manifestation of the evaluation
map.
\begin{lemma}\label{lem:varpi}
Let $\varpi : U(\sp_{2n}) \otimes_{U(\gp)} L_{r\lambda} \to V_{r\lambda}$ be the unique
$U(\sp_{2n})$--module homomorphism extending the natural $U(\gl_n)$--module
inclusion $L_{r\lambda} \to V_{r\lambda}$. Then (1) $\varpi$ is surjective, and (2) there
is a $\Sp_{2n}$--equivariant commutative diagram
\begin{equation}\xymatrix{
**[l]\Sp_{2n} \times^P V_{r\lambda}^*
\ar[r]^{\varpi^*}\ar[d]&**[r]
\Sp_{2n} \times^P [ U(\sp_{2n}) \otimes_{U(\gp)} L_{r\lambda} ]^* \ar[d]^{\rotatebox{-90}{$\simeq$}}  \\
**[l]X_n \times V_{r\lambda}^* \ar[r] &**[r] J^\infty(\LLL_{r\lambda}^*)
}\end{equation}
where the left vertical arrow is induced by the action map,
the right vertical arrow is the isomorphism \eqref{eq:verma-jet} of Lemma \ref{pro:verma-jet},
and the bottom horizontal arrow is induced by the evaluation map
$V_{r\lambda}^* \to \Gamma(X_n, \LLL_{r\lambda}^*)$.
\end{lemma}
\begin{proof}
Surjectivity of $\varpi$ follows from irreducibility of $V_{r\lambda}$.
Commutativity of the diagram is checked using Lemma \ref{lem:eval-action} and
an inspection of the identification \eqref{eq:verma-jet} of Proposition  \ref{pro:verma-jet}.
\end{proof}
\subsection{The resolution of the module $V_{r\lambda}^\ast$}\label{eqSubSecJanUltima}
\begin{proposition}\label{pro:verma-diff}
The identification of Corollary \ref{cor:equiv-diff} sends
$U(\sp_{2n})$--module homomorphisms $\delta$ such that
$
 0 \leftarrow V_{r\lambda} \leftarrow
U(\sp_{2n})\otimes_{U(\gp)}L_{r\lambda}
\xleftarrow{\delta}
U(\sp_{2n})\otimes_{U(\gp)}L_{\mu}
$
is exact to
equivariant differential operators $D$ such that
$ 0 \to V_{r\lambda}^* \to \Gamma(U, \LLL_{r\lambda}^*) \xrightarrow{D} \Gamma(U, \LLL_\mu^*) $
is exact for each nonempty, connected open $U \subset X_n$.
\end{proposition}
\begin{proof}
By Lemma \ref{lem:varpi} the isomorphism \eqref{eq:verma-jet} descends to
$\Sp_{2n}\times^P[\ker\varpi]^* \simeq \coker\left(X_n\times V^*_{r\lambda} \to J^\infty(\LLL^*_{r\lambda})\right)$,
where we use exactness properties of the associated bundle functor.
It follows that exactness of the sequence involving $\delta$ is equivalent
to exactness of
$$
0 \to X_n \times V_{r\lambda}^* \to J^\infty(\LLL_{r\lambda}^*) \xrightarrow{\tilde D}
J^\infty(\LLL_\mu^*)\, ,
$$
where $\tilde D$ is the vector bundle map induced by the differential operator $D$.
Hence, we only need to show that $V_{r\lambda^*}$ is the kernel of $D$ on the level
of local sections if
it is so on the level of formal sections. Clearly, in both cases
$V_{r\lambda^*}$ is contained in the kernel.
Now, assuming $V_{r\lambda}^*$ is the entire kernel of $\tilde D$ on $J^\infty_x(\LLL_{r\lambda}^*)$
for all $x \in X_n$, let $U \subset X_n$ be a nonempty connected open
subset, and consider $f \in \Gamma(U,\LLL_{r\lambda}^*)$
such that $Df=0$. We have that $j_x^\infty f$ is in the image of $V_{r\lambda}^*$ for all
$x \in U$, thus defining a map $v : U \to V_{r\lambda}^*$ such that $j_x^\infty f
= j_x^\infty v_x$, where $v_x \in V_{r\lambda}^*$ is understood as
a section of $\LLL_{r\lambda}^*$. The map $v$ is furthermore locally constant, and thus
constant by connectedness of $U$. Letting $\tilde v \in V_{r\lambda}^*$
denote its value, we have $j_x^\infty f = j_x^\infty \tilde v$ for all $x \in U$, and
thus $f=\tilde v|_U$.
\end{proof}

We shall now invoke the first (actually easy) step of the BGG resolution of $L_{r\lambda}$.
\begin{proposition}\label{pro:bgg}
Let $-\alpha \in \gh^*$ be the highest weight of the $U(\gl_n)$--module $S^2\R^n
\simeq \sp_{2n} / \gp$. There is an exact sequence
$
0 \leftarrow V_{r\lambda} \leftarrow
U(\sp_{2n})\otimes_{U(\gp)} L_{r\lambda}
\leftarrow
U(\sp_{2n})\otimes_{U(\gp)} L_{r\lambda - (r+1)\alpha}
$
of $U(\sp_{2n})$--module homomorphisms.
\end{proposition}
\begin{proof}
See J. Lepowsky, ``A generalization of the BGG resolution'' (1977) \cite{MR0476813}.
\end{proof}
\subsection{On the structure of the resolving operator}
By Proposition \ref{pro:verma-diff}, we \emph{know} that
the image of $V_{r\lambda}^*$ is the kernel of \emph{some} differential
operator from $\LLL_{r\lambda}^*$ to $\LLL_{r\lambda-(r+1)\alpha}^*$. But, what is this
operator, and how should we think of the target bundle? A good answer to the latter
question is the following.
\begin{lemma}
There is an  $\Sp_{2n}$--equivariant identification of
$\LLL_{r\lambda-(r+1)\alpha}^*$ with a direct summand of $S^{r+1} T^*X_n$.
\end{lemma}
\begin{proof}
Note first that since $L_{r\lambda}$ is one--dimensional and $-\alpha$
is the highest weight vector of $S^2\R^n$,
$L_{r\lambda-(r+1)\alpha}$ may be identified with the highest
weight summand $S^{2(r+1)}\R^n \subset S^{r+1} (S^2\R^n)$ as
a $U(\gl_n)$-module. With the trivial action of $S^2\R^{n*}$, this
extends to an identification of $U(\gp)$--modules, where $S^2\R^n$
is more naturally viewed as $\sp_{2n}/\gp$. We thus have
that the associated bundle $\LLL_{r\lambda-{r+1}\alpha}^*$ is
a direct summand of the associated bundle
$\Sp_{2n}\times^P [S^{r+1}(\sp_{2n}/\gp)]^* \simeq S^{r+1}T^*X_n$.
\end{proof}

Let us now proceed as before and decompose the modules into $\gl_n$--irreducibles.
Note first that the $U(\sp_{2n})$--module homomorphism
of Proposition \ref{pro:bgg} is uniquely determined by its restriction
to $L_{r\lambda-(r+1)\alpha}$, a $U(\gp)$-module homomorphism
$ \bar\delta :
L_{r\lambda - (r+1)\alpha}
\to
U(\sp_{2n})\otimes_{U(\gp)} L_{r\lambda}$.
On the domain side side of $\bar\delta$, we have already
seen that
$ L_{r\lambda - (r+1)\alpha} \simeq S^{2(r+1)}\R^n \otimes L_{r\lambda} $,
where $S^{2(r+1)}\R^n$ is the highest weight irreducible summand
in $S^{r+1} S^2\R^n$. On the codomain side,
\begin{equation}\label{eq:verma-to-sym}
U(\sp_{2n}) \otimes_{U(\gp)} L_\mu \simeq S^\bullet (S^2 \R^n) \otimes_\R L_\mu
\end{equation}
for any Verma module.
It follows that $\bar\delta$ induces a $U(\gl_n)$--module homomorphism
$
S^{2(r+1)}\R^n \otimes L_{r\lambda} \to S^\bullet(S^2\R^n)\otimes L_{r\lambda}
$.
Now, the domain is $\GL_n$--irreducible, and furthermore appears precisely
once as an irreducible summand in the codomain. Hence, $\bar\delta$ is necessarily
given by the $\GL_n$--equivariant inclusion. In particular, it factors through
$S^{r+1}(S^2\R^n) \otimes L_{r\lambda}$, thus showing that the corresponding differential
operator is of order $r+1$.\par
As a final step, we shall make it even more explicit by working in an adapted coordinate system.
Indeed, note that $S^2\R^n \subset \sp_{2n}$ is an abelian subalgebra, acting nilpotently in
the adjoint representation, and thus exponentiates to a subgroup $S^2\R^n \subset \Sp_{2n}$
with vector addition as the group operation. Since its Lie algebra is a complement
of $\gp$, it follows that the subgroup $S^2\R^n$ embeds as a dense open subset of $X_n
= \Sp_{2n}/P$; moreover, its inclusion map into $\Sp_{2n}$ may be then viewed as a local section
of the $P$--principal bundle $\Sp_{2n} \to X_n$. This allows us to trivialise all associated bundles
$
\LLL_\mu |_{S^2\R^n} \simeq S^2\R^n \times L_\mu
$,
so that a map $\LLL_\mu \to \LLL_\nu$ induced by a
$P$--equivariant map $\phi : L_\mu \to L_\nu$
identifies, upon restriction to $S^2\R^n \subset X_n$,
with $\mathrm{id}_{S^2\R^n} \times\phi$.
Now, the symmetry group $\Sp_{2n}$ is broken\footnote{Not to be confused with $P = \GL_n\ltimes S^2\R^{n*}$!} to
$S^2\R^n \rtimes \GL_n$, where
the factor $S^2\R^n$ acts on $S^2\R^n\subset X_n$ by translations.
This action is compatible with the trivialisations,
$\GL_n$ acting naturally on the $L_\mu$ factor.
The isomorphism of Proposition \ref{pro:verma-jet} becomes the
usual isomorphism
\begin{equation}\label{eq:sym-jet}
S^2\R^n \times [ S^\bullet(S^2\R^n) \otimes L_\mu]^* \simeq J^\infty(S^2\R^n \times \R)
\otimes L_\mu^*
\end{equation}
induced by the canonical affine connection $\nabla$ on $S^2\R^n$.
\begin{lemma}\label{lem:local-expr}
Let $\phi : U(\sp_{2n})\otimes_{U(\gp)} L_\mu \to U(\sp_{2n})\otimes_{U(\gp)} L_\nu$ be
a $U(\sp_{2n})$--module homomorphism, and denote by
$$ \bar\phi = \sum_i \bar\phi_i : L_\mu \to \bigoplus S^i(S^2\R^n)\otimes_\R L_\nu $$
its restriction to a $U(\gl_n)$--homomorphism.
Then, restricting to $S^2\R^n \subset X_n$ and using the standard trivialisations,
the differential operator corresponding to $\phi$ is the translation--invariant
$\GL_n$--equivariant operator
$F: C^\infty(S^2\R^n, L_\nu^*) \to C^\infty(S^2\R^n, L_\mu^*)$
defined by
$$
\langle Ff, h\rangle  = \sum_i \langle \bar\phi_i(h) , \nabla^{(i)} f\rangle
$$
where $f \in C^\infty(S^2\R^n, L_\nu^*)$,
$h \in L_\mu$
and
$\nabla^{(i)}f \in C^\infty(S^2\R^n, S^i (S^2\R^{n*}) \otimes L_\nu^*)$
is the tensor of $i^\textrm{th}$ derivatives.
\end{lemma}
\begin{proof}
The key point is the commutative diagram
\begin{equation*}\xymatrix{
\Sp_{2n} \times^P [ U(\sp_{2n}) \otimes_{U(\gp)} L_\mu ]^*
 \ar[r]^{\simeq}& J^\infty \LLL_\mu^* \\
(S^2\R^n \rtimes \GL_n) \times^{\GL_n} [ S^\bullet(S^2\R^n)\otimes L_\mu]^*
\ar[u]   \ar[r]^<<<<<<<{\simeq}   \ar[d]_{\rotatebox{90}{$\simeq$}}    & J^\infty\LLL^*_\mu|_{S^2\R^n}\ar[u] \ar[d]^{\rotatebox{-90}{$\simeq$}}     \\
S^2\R^n \times [ S^\bullet(S^2\R^n)\otimes L_\mu]^*
 \ar[r]^{\simeq} & J^\infty(S^2\R^n \times \R) \otimes L_\mu^*\, ,
}\end{equation*}
where: the top horizontal arrow is \eqref{eq:verma-jet}, the bottom horizontal
arrow is \eqref{eq:sym-jet}, the top left vertical arrow is induced
by \eqref{eq:verma-to-sym}, the bottom left vertical arrow
is induced the trivialising local section $X_n \supset S^2\R^n \to \Sp_{2n}$,
the top right vertical arrow is the restriction map,
and the bottom right vertical arrow is induced by the local trivialisation $\LLL^*_\mu|_{S^2\R^n}
\simeq S^2\R^n \times L_\mu^*$.
The inverse of the bottom horizontal arrow is induced by the map
$$
J^\infty(S^2\R^n\times\R) \xrightarrow{\sum_{i=0}^\infty \nabla^{(i)}} S^2\R^n \times
S^\bullet(S^2\R^n)^*
$$
tensored with $\mathrm{id}_{L_\mu^*}$, whence the desired formula follows.
\end{proof}
Now we are in position of proving the main result of this section.
\begin{theorem}\label{thm:final-bgg}
Identify $S^2\R^n$ with a dense open subset of $X_n$. Let $\nabla^{(i)} :
C^\infty(S^2\R^n) \to C^\infty(S^2\R^n, S^i (S^2\R^{n*})$ denote the
$i^\textrm{th}$ partial derivatives map induced by the canonical affine connection on
$S^2\R^n$. Then $r^\textrm{th}$ degree   hypersurface sections of $X_n \subset \p V_\lambda$
passing through $S^2\R^n$ are precisely the zero--loci of non--trivial functions
in the kernel of the differential operator
$$
C^\infty(S^2\R^n) \xrightarrow{\nabla^{(r+1)}} C^\infty(S^2\R^n, S^{r+1} (S^2\R^{n*}))
\xrightarrow{\mathrm{pr}} C^\infty(S^2\R^n, S^{2(r+1)}\R^{n*}).
$$
\end{theorem}
\begin{proof}
Using Lemma \ref{lem:local-expr},
one only needs to choose an identification $L_{r\lambda} \simeq \R$ as one--dimensional
vector spaces.
\end{proof}
A little bit more can be said about the functions on $S^2\R^n$ in the kernel
of $\mathrm{pr} \circ \nabla^{(r+1)}$. A simple computation with weights
shows that the restriction
of $\varpi:U(\sp_{2n})\otimes_{U(\gp)}\otimes L_{r\lambda} \to V_{r\lambda}$ to
$S^{\le \ell} (S^2\R^n) \otimes L_{r\lambda}$ is surjective if and only if $\ell \ge nr$;
furthermore, $S^\ell (S^2\R^n) \otimes L_{r\lambda}$ is contained in $\ker\varpi$
for all $\ell > nr$. Dualising these statements one sees that the solutions
of $\mathrm{pr}\circ\nabla^{(r+1)}$ are polynomials of degree at most $nr$,
and that this bound is saturated. Observe that, for fixed $r>0$, the degree
of solutions grows linearly with the dimension $n$, while the order of the differential operator
is constant.
\subsection{Tensorial obstructions to  complete exceptionality}\label{subsecLinkSecJanConLaPrecedente}
In general, we do not know an explicit intrinsic characterisation of $r^\textrm{th}$ degree   hypersurface
sections as submanifolds, in the same spirit as Theorem
\ref{th.main.1}. However, for $r=1$ (i.e., hyperplane sections) and any $n\ge 2$ we
may at least construct a tensorial obstruction. Namely, let $\Sigma \subset X_n$
be a codimension $1$
submanifold, and write $F^\ell \subset\Gamma(\Sigma, \LLL_\lambda^*)$ for the
sub--module of sections of $\LLL_\lambda^*$ vanishing to $(\ell-1)$-st order along $\Sigma$.
Recall that $F^1/F^2$ is naturally identified
with the module $\Gamma(\Sigma, N^*\Sigma\otimes\LLL_\lambda^*)$
of global sections of the twisted conormal bundle.
We then have a commutative diagram
\begin{equation*}\xymatrix{
**[l]F^1/F^3 \ar[r] \ar[d]& **[r]\Gamma(\Sigma, S^2 T^*X_n|_\Sigma \otimes \LLL_\lambda^*|_\Sigma)\ar[d] \\
**[l]F^1/F^2 \ar[r]^{\Phi_\Sigma}&**[r]\Gamma(\Sigma, S^{2} T^*\Sigma \otimes \LLL_\lambda^*|_\Sigma)\, ,
}\end{equation*}
where the top horizontal arrow
is induced by our differential operator $\Gamma(\LLL_\lambda^*) \to \Gamma(S^2T^*X_n\otimes\LLL_\lambda^*)$,
while the right vertical arrow is the pullback map along $\Sigma \to S^2\R^n$.
One checks (e.g., using Theorem \ref{thm:final-bgg}) that the bottom horizontal arrow is $C^\infty$--linear, thus corresponding
to a tensor
$ \Phi_\Sigma \in \Gamma(\Sigma, S^2 T^*\Sigma \otimes N\Sigma) $,
whose vanishing is a \emph{necessary} condition for $\Sigma$ to be a hyperplane section.
For $n=2$, one may identify $\Phi_\Sigma$ with
the trace--free second fundamental form of $\Sigma$ with respect
to the invariant conformal structure on $X_2$. Thus, in the notation of Theorem
\ref{th.main.1}, vanishing of $\Phi_\Sigma$ is equivalent to the
condition $\II_\Sigma^{T_2} \in S^2_{T_2}|_\Sigma$ over the dense open $S^2\R^2 \subset X_2$,
where $T_2$ is well--defined as an actual metric tensor. For $n=3$, again
over $S^2\R^3 \subset X_3$, one may identify $\Phi_\Sigma$ with
$\II_\Sigma^{X \lrcorner T_3}$ modulo $S^2_{T_3}|_\Sigma$ for any
\emph{translation} vector field $X \in S^2\R^2$ of maximal rank. Indeed,
the tensor $T_3$ may be chosen constant with respect to the
affine connection $\nabla$ on $S^2\R^2$, whence
$\nabla$ is the Levi--Civita connection for $X \lrcorner T_3$. Then, by Theorem
\ref{th.main.1}, vanishing of $\Phi_\Sigma$ is again equivalent
to the condition $\II_\Sigma^{X \lrcorner} T_3 \in S^2_{T_3}|_\Sigma$
for any $X \in \mathrm{sym}_{\mathrm{max}}(T_3)$. Hence,
for $n=2,3$ we conclude that vanishing of $\Phi_\Sigma$ is also a \emph{sufficient}
condition for $\Sigma$ to be a hyperplane section.

\subsection{The Lagrangian Grassmannian as a rational projective variety}\label{secLinkWithProjGeom}

Observe that $X_n$ has the same dimension as $\p W_0$, where
$ 
W_0:=\R \oplus S^2\R^{n\ast}
$ 
 is a $\gl_n$--submodule of $V_\lambda$. Let us denote by $W_1$ the sum of the remaining irreducible pieces of    $V_\lambda$.
 \begin{theorem}
 The natural projection $\p V_\lambda\longrightarrow \p W_0$ restricts to a birational map
 \begin{equation}\label{eqProjRazionale}
 X_n\longrightarrow \p W_0\, .
 \end{equation}
 \end{theorem}
\begin{proof}
It will be carried out by constructing an explicit inverse for \eqref{eqProjRazionale}. To this end, we must observe that any bilinear form $\alpha\in S^2\R^{n\ast} $ can be extended naturally to a bilinear form on $\Lambda^k\R^n$, and that such an extension, say $\alpha^{(k)}$ belongs in fact to $S^2_0\Lambda^k\R^{n\ast} $.\par
The desired inverse is then given by the rational map
\begin{eqnarray}
 \p W_0&\stackrel{p}{\longrightarrow} &\p(W_0\oplus W_1)=\p V_\lambda\, ,\label{eqPiWuZero}\\
\left[\lambda:\alpha\right]&\longmapsto & \left[\lambda^n:\lambda^{n-1}\alpha:\lambda^{n-2}\alpha^{(2)}:\cdots:\alpha^{(n)}\right]\, ,\nonumber
\end{eqnarray}
whose image lies in $X_n$.
\end{proof}
 \begin{remark}
 By restricting \eqref{eqPiWuZero} to the affine neighborhood $S^2\R^{n\ast}\equiv\{\lambda=1\}\subset \p W_0$, one gets precisely the special affine neighborhood $U\subset X_n$ defined earlier.
 \end{remark}
\begin{corollary}
The Lagrangian Grassmannian is rational.
\end{corollary}
 The   map $p:\p W_0\longrightarrow \p V_\lambda$ defined as in \eqref{eqPiWuZero} is of degree $n$, and as such it factors through the Veronese map, i.e., $p=h\circ v_n$, where $v_n$ is the Veronese map, and $h$ is linear, viz.
$ 
\p W_0  \stackrel{v_n}{\longrightarrow} \p S^n W_0  \stackrel{h}{\longrightarrow} \p V_\lambda
$. 
Dually,
$ 
\O_{\p V_\lambda}(r) \stackrel{h^*}{\longrightarrow} \O_{\p W_0}(nr)
$ 
and, by taking sections over $U$,
\begin{equation}\label{eqUnPoFascista}
\Gamma(\O_{\p V_\lambda}(r),U) \stackrel{h^*}{\longrightarrow} S^{nr}S^2\R^n\, .
\end{equation}
In other words, the Lagrangian Grassmannian $X_n$ is covered by the Veronese variety $v_n(\p^n)$ through a linear projection $h$.
 \begin{theorem}\label{thUltimoUnPoAlgebrista}
 The image of the morphism $h^*$ in \eqref{eqUnPoFascista} is made precisely by the polynomials of degree $n$ on $U$, which corresponds to degree--$r$ hypersurface sections. Moreover,  \eqref{eqUnPoFascista} becomes an exact sequence if the total symmetrisation $S^{nr}S^2\R^n\longrightarrow S^{2nr} \R^n$ is appendend.
 \end{theorem}
 \begin{proof}
 Follows from Theorem \ref{thm:final-bgg}.
 \end{proof}

 \subsection{Canonical equations on $X_n$}\label{secLinkWithJets}

It should be clear that the main results obtained so far, namely Theorem \ref{th.main.1},   Theorem \ref{thm:final-bgg} and Theorem \ref{thUltimoUnPoAlgebrista}, have something  in common: they all allow to recognise special hypersurfaces in $X_n$ by manipulating the natural structures associated with $X_n$.  The natural place for studying the differential invariants of the   hypersurfaces in $X_n$ is the jet space
$
 J^r:=J^r(X_n , d_n-1 )
$, 
where $r$ is the order of the invariant (i.e., hyperplane sections will be characterised by the vanishing of a first order invariant), and $d_n :=\dim X_n=\frac{n(n+1)}{2}$.
\begin{theorem}
$J^r$ is equipped with a canonical vertical distribution $\E$, which is tangent to the equation of degree--$r$ hyperplane sections.
\end{theorem}
\begin{proof}
We just deal with the case $r=1$, since the general one is formally analogous. Recall the   canonical identifications
$
J^1=\p T^* X_n=\p(S^2 \L)
$, 
and that the contact plane on $J^1$ is given by
\begin{equation*}
\CC_{[q]}:=\ker q \oplus T_{[q]}(\p S^2 L)\subset S^2 L^*\oplus  T_{[q]}(\p S^2 L) =T_{[q]}J^1\, ,
\end{equation*}
where $L\in X_n$ is the projection of $[q]\in J^1$. But
\begin{equation*}
T_{[q]}(\p S^2 L)=(\ker q)^*\otimes\frac{  S^2 L}{\ker q}\equiv (\ker q)^*\, ,
\end{equation*}
so that
$
\CC_{[q]}:= \ker q \oplus (\ker q)^*
$ 
is the canonical symplectic $2(n-1)$--dimensional space $\R^{n-1}\oplus\R^{n-1\ast}$.
\par
According, the fibre $J^2_{[q]}$ of the projection $J^2\to J^1$ is given by\footnote{We denote by $X^\circ_n$ the ``big cell'' in $X_n$, i.e., the standard affine neighborhood.}
\begin{equation*}
J^2_{[q]}:=\LL(n-1,\CC_{[q]})^\circ=\LL(n-1,2(n-1))^\circ\equiv X_{n-1}^\circ\, ,
\end{equation*}
which is an affine space modelled over
$
S^2(\ker q)^*\equiv S^2\R^{n-1}
$. 
   But since there is a natural restriction map
   \begin{equation}\label{eqSimmetrizzazionmeFinale}
   S^2(S^2L^*)\longrightarrow S^2 (\ker q)^*
   \end{equation}
  and  $S^2(S^2L^*)$ contains the special subspace $S^4L^*$, then there is also a distinguished subspace
   $S^2_0 (\ker q)^*$ in $S^2 (\ker q)^*$. \par
 Since the linear models of the fibres of $J^2\to J^1$ contain a canonical linear subspace, the bundle $J^2$ is equipped with a canonical vertical distribution. The result follow from Theorem \ref{thm:final-bgg}, by comparing \eqref{eqSimmetrizzazionmeFinale} with $\mathrm{pr}\circ\nabla^{(2)}$.

\end{proof}
The so--obtained vertical distribution can be regarded as an ``infinitesimal" description of the (generalised) condition of complete exceptionality.

\subsection*{Conclusions}
All reasoning carried out here stemmed from the peculiarity of the  tangent geometry of the Lagrangian Grassmannian. In one way or another, we always relied on the fundamental isomorphism \eqref{eq.iso.solito}. In small dimensions and low degree, it was easy to realise that among the natural tensorial identities following from \eqref{eq.iso.solito}, there are the conditions of complete exceptionality which, in turn, can be recast in the familiar terms of conformal (possibly ``trivalent'') geometry. Even the result  obtained in   Theorem \ref{thm:final-bgg} could be guessed---so to speak---by analogy, but proving its validity necessarily required such a heavy machinery as the BGG resolution.

\subsection*{Acknowledgements}

Important hints and suggestions came from David Calderbank, Willie Wong, V\'it Tu\v{c}ek, Ciro Ciliberto, and Pawe\l\ Nurowski.\footnote{Many thanks go to the MathOverflow forum    for providing an effective platform of knowledge--sharing.} The work of Jan Gutt was supported in part by the Polish National Science Center (NCN) via DEC--2013/09/B/ST1/01799. The research of Gianni Manno has been partially supported by the project   ``FIR (Futuro in Ricerca) 2013 -- Geometria delle equazioni differenziali''. The research of Giovanni Moreno has been partially supported by the Marie Sk\l odowska--Curie fellowship SEP--210182301 ``GEOGRAL". Both Gianni Manno and Giovanni Moreno are members of G.N.S.A.G.A. of I.N.d.A.M.



\begin{thebibliography}{10}

\bibitem{MR1875343}
S.~I. Agafonov and E.~V. Ferapontov.
\newblock Systems of conservation laws of {T}emple class, equations of
  associativity and linear congruences in {$\bold P^4$}.
\newblock {\em Manuscripta Math.}, 106(4):461--488, 2001.

\bibitem{MR2985508}
Dmitri~V. Alekseevsky, Ricardo Alonso-Blanco, Gianni Manno, and Fabrizio
  Pugliese.
\newblock Contact geometry of multidimensional {M}onge-{A}mp\`ere equations:
  characteristics, intermediate integrals and solutions.
\newblock {\em Ann. Inst. Fourier (Grenoble)}, 62(2):497--524, 2012.

\bibitem{MR1139843}
Guy Boillat.
\newblock Sur l'\'equation g\'en\'erale de {M}onge-{A}mp\`ere \`a plusieurs
  variables.
\newblock {\em C. R. Acad. Sci. Paris S\'er. I Math.}, 313(11):805--808, 1991.

\bibitem{MR1194520}
Guy Boillat.
\newblock Sur l'\'equation g\'en\'erale de {M}onge-{A}mp\`ere d'ordre
  sup\'erieur.
\newblock {\em C. R. Acad. Sci. Paris S\'er. I Math.}, 315(11):1211--1214,
  1992.

\bibitem{Boillat_et_al}
Guy Boillat, Constantin~M. Dafermos, Peter~D. Lax, and Tai-Ping Liu.
\newblock {\em Recent Mathematical Methods in Nonlinear Wave Propagation:
  Lectures given at the 1st Session of the Centro Internazionale Matematico
  Estivo ... Mathematics / C.I.M.E. Foundation Subseries)}.
\newblock Springer, 1996.

\bibitem{MR0481580}
Guy Boillat and Tommaso Ruggeri.
\newblock Characteristic shocks: completely and strictly exceptional systems.
\newblock {\em Boll. Un. Mat. Ital. A (5)}, 15(1):197--204, 1978.

\bibitem{MR571041}
Giovanni Crupi and Andrea Donato.
\newblock A class of conservative and hyperbolic completely exceptional
  equations which are compatible with a supplementary conservation law.
\newblock {\em Atti Accad. Naz. Lincei Rend. Cl. Sci. Fis. Mat. Natur. (8)},
  65(3-4):120--127 (1979), 1978.

\bibitem{DonatoRamgulamRogers}
Andrea Donato, Usha Ramgulam, and Colin Rogers.
\newblock The (3+1)-dimensional monge-ampère equation in discontinuity wave
  theory: Application of a reciprocal transformation.
\newblock {\em Meccanica}, 27(4):257--262, 1992.

\bibitem{MR1292999}
Andrea Donato and Giovanna Valenti.
\newblock Exceptionality condition and linearization procedure for a third
  order nonlinear {PDE}.
\newblock {\em J. Math. Anal. Appl.}, 186(2):375--382, 1994.

\bibitem{9780471742999}
Robert~W. Fox, Philip~J. Pritchard, and Alan~T. McDonald.
\newblock {\em Introduction to Fluid Mechanics}.
\newblock Wiley, 2008.

\bibitem{2014arXiv1405.5198J}
G.~R. {Jensen}.
\newblock {Dupin hypersurfaces in Lie sphere geometry}.
\newblock {\em ArXiv e-prints}, May 2014.

\bibitem{MR2352610}
Alexei Kushner, Valentin Lychagin, and Vladimir Rubtsov.
\newblock {\em Contact geometry and non-linear differential equations}, volume
  101 of {\em Encyclopedia of Mathematics and its Applications}.
\newblock Cambridge University Press, Cambridge, 2007.

\bibitem{MR0068093}
P.~D. Lax.
\newblock The initial value problem for nonlinear hyperbolic equations in two
  independent variables.
\newblock In {\em Contributions to the theory of partial differential
  equations}, Annals of Mathematics Studies, no. 33, pages 211--229. Princeton
  University Press, Princeton, N. J., 1954.

\bibitem{MR0476813}
J.~Lepowsky.
\newblock A generalization of the {B}ernstein-{G}elfand-{G}elfand resolution.
\newblock {\em J. Algebra}, 49(2):496--511, 1977.

\bibitem{ArticoloAntipatetico}
Gianni Manno and Giovanni Moreno.
\newblock Meta-symplectic geometry of {$3^{\rm rd}$} order {M}onge-{A}mp\`ere
  equations and their characteristics.
\newblock {\em SIGMA Symmetry Integrability Geom. Methods Appl.}, 12:032, 35
  pages, 2016.

\bibitem{ArticoloConPawelEKatia}
Gianni Manno, Giovanni Moreno, Pawel Nurowski, and Katja Sagerschnig.
\newblock in preparation.

\bibitem{MR1606791}
Francesco Oliveri.
\newblock Linearizable second order {M}onge-{A}mp\`ere equations.
\newblock {\em J. Math. Anal. Appl.}, 218(2):329--345, 1998.

\bibitem{9780125267403}
Barrett O'Neill.
\newblock {\em Semi-Riemannian Geometry With Applications to Relativity, 103,
  Volume 103 (Pure and Applied Mathematics)}.
\newblock Academic Press, 1983.

\bibitem{RuggeriWave}
T.~Ruggeri.
\newblock Recent results on wave propagation in continuum models.
\newblock In G.P. Galdi, editor, {\em Stability and Wave Propagation in Fluids
  and Solids}, volume 344 of {\em CISM International Centre for Mechanical
  Sciences}, pages 105--154. Springer Vienna, 1995.

\bibitem{Ruggeri01061978}
T.~Ruggeri and A.~Strumìa.
\newblock Exceptional waves and characteristic shocks on non-linear
  relativistic strings.
\newblock {\em Progress of Theoretical Physics}, 59(6):2121--2132, 1978.

\bibitem{Temple}
B.~Temple.
\newblock Systems of conservation laws with coinciding shock and rarefaction
  curves.
\newblock In {\em Nonlinear partial differential equations ({D}urham, {N}.{H}.,
  1982)}, volume~17 of {\em Contemp. Math.}, pages 143--151. Amer. Math. Soc.,
  Providence, R.I., 1983.

\bibitem{MR2876965}
Dennis The.
\newblock Conformal geometry of surfaces in the {L}agrangian {G}rassmannian and
  second-order {PDE}.
\newblock {\em Proc. Lond. Math. Soc. (3)}, 104(1):79--122, 2012.

\bibitem{MR722524}
Keizo Yamaguchi.
\newblock Contact geometry of higher order.
\newblock {\em Japan. J. Math. (N.S.)}, 8(1):109--176, 1982.

\end{thebibliography}


\end{document}